\documentclass[11pt, oneside]{amsart}   	% use "amsart" instead of "article" for AMSLaTeX format
\usepackage{etex}
\usepackage{geometry}                		% See geometry.pdf to learn the layout options. There are lots.
\usepackage{graphicx}				% Use pdf, png, jpg, or eps with pdflatex; use eps in DVI 	
\usepackage{amssymb}
\usepackage{amsmath}
\usepackage{amsthm}
\usepackage[all]{xy}
\usepackage{enumerate}
\usepackage{todonotes}

\theoremstyle{definition}
\newtheorem{thm}{Theorem}
\newtheorem{prop}{Proposition}
\newtheorem{lem}{Lemma}
\newtheorem{defn}{Definition}

\newtheorem{definition}[thm]{Definition}

\usepackage{tikz}
\usetikzlibrary{calc,patterns,decorations.pathmorphing,decorations.markings}
\usepackage{tikz-cd}
\usepackage{circuitikz}

\def\id{\mathop{\mathrm{id}}\nolimits}

\usepackage[numbers, sort]{natbib}

\usepackage{multicol}

\newcommand{\pder}[2]{\frac{\partial #1}{\partial #2}}

\newcommand{\Span}[0]{\text{span}}

\title{Variational integrators for interconnected Lagrange--Dirac systems}
\author{Helen Parks}
\author{Melvin Leok}
\address{Department of Mathematics, University of California, San Diego, 9500 Gilman Drive, La Jolla, California, USA.}
\email{parks.helen@gmail.com, mleok@math.ucsd.edu}
\begin{document}
\maketitle

\begin{abstract}
Interconnected systems are an important class of mathematical models, as they allow for the construction of complex, hierarchical, multiphysics, and multiscale models by the interconnection of simpler subsystems. Lagrange--Dirac mechanical systems provide a broad category of mathematical models that are closed under interconnection, and in this paper, we develop a framework for the interconnection of discrete Lagrange--Dirac mechanical systems, with a view towards constructing geometric structure-preserving discretizations of interconnected systems. This work builds on previous work on the interconnection of continuous Lagrange--Dirac systems \cite{JaYo2014} and discrete Dirac variational integrators \cite{LeOh2011}. We test our results by simulating some of the continuous examples given in \cite{JaYo2014}.
\end{abstract}

\section{Introduction}
This work is motivated in part by a desire to develop a geometric structure-preserving simulation framework with which to model control systems by using interconnections. By interconnection, we mean a Dirac structure, which is a generalization of symplectic and Poisson structures that can geometrically encode the nonholonomic constraints between subsystems. The need for robust control of mechanical systems is perhaps one of the most common reasons for viewing a system in terms of interconnections. We have a plant system whose behavior we wish to control, so it must be mechanically or electrically joined to a controller system. Hence, we have an interconnected system. Since the controlling device is often itself a mechanical system, we have the interconnection of two mechanical systems, and we can begin to study the structure of the interconnected, controlled system as it relates to the structures of the starting plant and controller. The field of port-Hamiltonian systems and the associated feedback stabilization control paradigm, Interconnection and Damping Assignment - Passivity-Based Control (IDA-PBC), undertakes just such an approach and is already a very well-established methodology with an extensive range of results \cite{VdS2006, DuMaStBr2009}, and which can be viewed as being dual to the method of controlled Lagrangians~\cite{Ch2014}.

As the name suggests, port-Hamiltonian systems adopt a Hamiltonian perspective on interconnected systems. In \citet{YoMa2006a, YoMa2006b}, Lagrange--Dirac mechanics were developed as a way of understanding the implicit systems central to port-Hamiltonian systems from the Lagrangian perspective. That aim is rooted partially in the natural desire to understand implicit systems from both classical perspectives. It also moves toward the goal of numerically simulating interconnections and control by interconnection using structured computational methods via variational integrators. Variational integrators have been developed for a broad class of problems, including, \citet{LaWe2006, LeZh2011} for Hamiltonian systems; \citet{FeMaOrWe2003} for nonsmooth problems with collisions; \citet{MaPaSh1998, LeMaOrWe2003} for Lagrangian PDEs; \citet{CoMa2001,McPe2006, FeZe2005} for nonholonomic systems; \citet{BoOw2009, BoOw2010} for stochastic Hamiltonian systems; \citet{LeLeMc2007, LeLeMc2009, BoMa2009} for problems on Lie groups and homogeneous spaces; \citet{LeOh2011} for Lagrange--Dirac mechanical systems. However, most of the work on variational integrators has adopted the Lagrangian as opposed to the Hamiltonian perspective, and this is the approach that we will adopt as well in this paper.

%While variational integrators cover a wide range of numerical schemes but are most often derived from the Lagrangian perspective.

The next steps were taken in \citet{JaYo2014}, which develop continuous interconnections of Lagrange--Dirac systems, and in \citet{LeOh2011}, where variational integrators were extended to the Lagrange--Dirac case. The discrete Lagrange--Dirac mechanics introduced in \citet{LeOh2011} can be viewed as a generalization of the discrete nonholonomic mechanics introduced by \citet{CoMa2001} to the setting of degenerate systems, which yields an implicit version of the discrete equations of motion. This implicit system of equations is analogous to rewriting the second-order Lagrange--d'Alembert equations of continuous nonholonomic mechanics \cite{Bl2003} in first-order form by introducing the Legendre transformation. In addition, it also provides an alternative derivation of the discrete equations of motion in terms of an associated discrete Dirac structure. In this paper, we discretize the interconnections of \cite{JaYo2014} in accordance with the framework laid out in \cite{LeOh2011}, and describe how this is achieved both in terms of discrete variational principles and discrete Dirac structures.

While our study of interconnected systems has very specific roots, we have abstracted our way to general interconnections (following \cite{JaYo2014}) and believe that our results have useful applications outside the realm of plant/controller interconnection. It is natural to approach the modeling of a large, complex system by breaking it into smaller, more easily understood components. The full system can then be modeled as the interconnection of several simpler, component-wise models. Sometimes our engineering objectives themselves are modular, such as with a robot in need of several different appendages, each with a specific function. Interconnection through the use of Dirac structures provides a mathematical framework for modeling such modular designs in a natural fashion, and may reduce the incremental cost of constructing full system models when the appendages are changed, since the model of the appendage subsystem can be swapped out without the need to modify the rest of the mathematical model.

More generally, this can allow the reusability and exchange of commonly used model subsystems, and provide the basis for constructing more complicated models by assembling and interconnecting model subsystems, instead of constructing each new model monolithically from scratch. 
%alter the models of each appendage individually with only minimal effort to reconstruct a full system model. 
This also naturally leads to a framework for developing parallel and distributed numerical implementations of such structure-preserving simulations. As with all modular, parallel, and distributed computations, the efficiency of such a modeling and simulation approach is dependent on choosing a decomposition of the full system into component subsystems that involve minimal coupling between subsystems, otherwise the interconnection and communications overhead can outweigh the benefits of decomposing the model and simulation.

\section{Background}

\subsection{Dirac structures and Langrange--Dirac mechanics}

Dirac structures are the simultaneous generalization of symplectic and Poisson structures, and can encode Dirac constraints that arise in degenerate Lagrangian systems, interconnected systems, and nonholonomic systems, and thereby provide a unified geometric framework for studying such problems. We begin with a review of Dirac structures and their role in Lagrange--Dirac mechanics. Then, we revisit the continuous interconnection process. 

\subsubsection{Dirac structures}

Let $V$ be a finite-dimensional vector space with dual $V^*$. Denote the natural pairing between $V$ and $V^*$ by $\langle \cdot, \cdot \rangle$, and define the symmetric pairing $\langle\langle \cdot, \cdot \rangle\rangle$ on $V\oplus V^*$ by 
\begin{equation}
\langle \langle(v_1,\alpha_1),(v_2,\alpha_2)\rangle\rangle = \langle \alpha_1, v_2 \rangle + \langle \alpha_2, v_1 \rangle
\end{equation}
for $(v_1,\alpha_1), (v_2,\alpha_2) \in V \oplus V^*$. A Dirac structure on $V$ is a subset $D \subset V\oplus V^*$ such that $D = D^\perp$ with respect to $\langle\langle \cdot , \cdot \rangle\rangle$. Given a subspace $\Delta\subset V$ and its annihilator $\Delta^\circ= \{ \alpha \in V^* \ | \ \langle \alpha,v \rangle = 0 \ \text{for all} \ v\in\Delta \}\subset V^*$, we can construct $\Delta\oplus\Delta^\circ\subset V\oplus V^*$, which is an example of a Dirac structure.

Now, let $M$ be a smooth manifold. Denote by $TM\oplus T^*M$ the Pontryagin bundle over $M$, where the fiber over $x\in M$ is $T_xM \oplus T_x^*M$. Then, a \emph{Dirac structure on $M$} is a subbundle $D\subset TM\oplus T^*M$ such that every fiber $D(x)$ is a Dirac structure on $T_xM$. An \emph{integrable Dirac structure} has the additional property, $\langle\pounds_{X_1} \alpha_2, X_3\rangle+\langle\pounds_{X_2} \alpha_3, X_1\rangle+\langle\pounds_{X_3} \alpha_1, X_2\rangle=0,$
for all pairs of vector fields and one-forms $(X_1,\alpha_1),(X_2,\alpha_2),(X_3,\alpha_3)\in D$, where $\pounds_X$ is the Lie derivative. This generalizes  the condition that the symplectic two-form is closed, or that the Poisson bracket satisfies Jacobi's identity. For the purposes of this paper, we will not assume that a Dirac structure satisfies the integrability condition, since it does not hold for Dirac structures that incorporate non-integrable or nonholonomic constraints. It should be noted that such non-integrable Dirac structures are sometimes referred to in the literature as almost-Dirac structures.

Every manifold Dirac structure $D$ has an associated distribution defined by
\begin{equation}
\Delta_D(x) = \{ v \in T_xM \ | \ (v,\alpha) \in D(x) \ \text{for some} \ \alpha\in T_x^* M\}.
\end{equation}
The Dirac structure $D$ also defines a bilinear map on $\Delta_D$,
\begin{equation}
\omega_{\Delta_D}(v,u) = \langle \alpha_v, u \rangle,
\end{equation}
for any $\alpha_v$ such that $(v,\alpha_v) \in D(x)$ and any $u \in \Delta_D(x)$. The two-form $\omega_{\Delta_D}$ is well-defined on $\Delta_D$ even if there exist multiple such $\alpha_v$ since $D = D^\perp$ with respect to the symmetric pairing above.

Conversely, given a two-form $\omega$ on $M$ and a regular distribution $\Delta\subset TM$, we can define a Dirac structure $D$ on $M$ fiber-wise as
\begin{equation}
D(x) = \{(v,\alpha) \in T_xM \oplus T_x^*M \ | \ v \in \Delta(x) \ \text{and} \ \langle \alpha, u\rangle = \omega_x(v,u) 
\text{ for all} \ u \in \Delta(x)\}.
\label{fiberwise_Dirac_def}
\end{equation}
Clearly, in this case $\Delta_D = \Delta$ and $\omega_{\Delta_D} = \omega |_{\Delta_D}$. We use this idea to connect Dirac structures with constraint distributions.

\subsubsection{Induced Dirac structures}

Dirac structures are especially relevant in the case of Lagrangian systems with linear nonholonomic constraints, i.e. constraints of the form $\omega^a(q)\cdot\dot{q} = 0$, $a = 1,\dots,m$, where $\omega^a$ are one-forms on $Q$. The interested reader is referred to \citet{Bl2003} for a more in-depth discussion of nonholonomic mechanics and constraints. Such constraints can be equivalently expressed using the regular distribution $\Delta_Q\subset TQ$ defined by $\Delta_Q(q) = \cap_a \ker(\omega^a(q))$. Thus, the annihilator codistribution of $\Delta_Q$ is given by $\Delta_Q^\circ(q) = \Span\{\omega^a(q)\}$. The constraints are then written $\dot{q} \in \Delta_Q(q)$ or simply $\dot{q} \in \Delta_Q$. Nonholonomic constraints such as these cause the motion on $T^*Q$ to be pre-symplectic rather than symplectic. The Dirac structure induced by $\Delta_Q$ gives a precise description of this pre-symplectic structure. Note that we may also have primary constraints on $T^*Q$ if $L$ is degenerate.

The constraints $\Delta_Q$ induce a Dirac structure on $T^*Q$ as follows. From $\Delta_Q$, define $\Delta_{T^*Q} \subset TT^*Q$ as
\begin{equation}
\Delta_{T^*Q} = (T\pi_Q)^{-1}(\Delta_Q)
\end{equation}
for the canonical projection $\pi_Q : T^*Q \to Q$ and its tangent lift $T\pi_Q$. This definition will become clearer in the next section, when we discuss the representation in local coordinates. We now apply the construction described in \eqref{fiberwise_Dirac_def} using $\Delta_{T^*Q}$ and the canonical symplectic form $\Omega$ on $T^*Q$. This gives the following fiber-wise definition of $D_{\Delta_Q}$, the Dirac structure on $T^*Q$ induced by the constraint distribution $\Delta_Q$.
\begin{equation}
\begin{aligned}
D_{\Delta_Q}(q,p) &= \{(v,\alpha)\in T_{(q,p)}T^*Q \oplus T_{(q,p)}^*T^*Q \ | \ v \in \Delta_{T^*Q}(q,p) \ \text{and} \\
&\qquad \langle \alpha, u \rangle = \Omega(v,u) \ \text{for all} \ u\in \Delta_{T^*Q}(q,p)\}.
\label{induced_Dirac_def}
\end{aligned}
\end{equation}

\subsubsection{Canonical local coordinate expressions}

It will be useful to have expressions for $\Delta_{T^*Q}, \Delta_{T^*Q}^\circ$, and $D_{\Delta_Q}$ in terms of local canonical coordinates. Let $V$ be a model vector space for the configuration manifold $Q$, and let $U\subset V$ be a chart around $q\in Q$. Then, we have the following local representations near $q$,
\begin{align*}
TQ \mapsto U\times V,\\
T^*Q \mapsto U\times V^*,\\
TTQ \mapsto (U\times V)\times(V\times V),\\
TT^*Q \mapsto (U\times V^*)\times (V\times V^*),\\
T^*T^*Q \mapsto (U\times V^*)\times (V^*\times V).
\end{align*}
In these coordinates $\pi_Q : (q,p)\mapsto q$ and $T\pi_Q: (q,p,\delta q, \delta p)\mapsto (q,\delta q)$, so that
\begin{equation}
\Delta_{T^*Q} = \{(q,p,\delta q, \delta p) \in T_{(q,p)}T^*Q \ | \ (q,\delta q) \in \Delta_Q\},
\end{equation}
and the annihilator distribution is given by
\begin{equation}
\Delta_{T^*Q}^\circ(q,p) = \{(q,p,\alpha_q,\alpha_p) \in T_{(q,p)}^*T^*Q \ | \ (q,\alpha_q) \in \Delta_Q^\circ \ \text{and} \ \alpha_p = 0\}.
\end{equation}
As indicated above, any $v\in T_{(q,p)}T^*Q$ has two coordinate components. We will write these as $(\delta q, \delta p)$ in the abstract case or $(v_q,v_p)$ when referring to a particular $v$. Similarly, we will write $\alpha = (\alpha_q,\alpha_p)$ for $\alpha \in T_{(q,p)}^*T^*Q$. In this notation, $\Omega(v,u) = v_q\cdot u_p - v_p \cdot u_q$. So the condition $\langle \alpha, u\rangle = \Omega(v,u)$ for all $u \in \Delta_{T^*Q}$ translates to $(\alpha_q+v_p,\alpha_p-v_q) \in \Delta_{T^*Q}^\circ$. Thus, the induced Dirac structure in \eqref{induced_Dirac_def} has the coordinate expression
\begin{equation}
\begin{aligned}
D_{\Delta_Q}(q,p) &= \{(v_q,v_p,\alpha_q,\alpha_p)\in T_{(q,p)}T^*Q \oplus T_{(q,p)}^*T^*Q \ | \ v_q \in \Delta_{Q}(q),\\ 
&\qquad\alpha_p = v_q, \ \text{and} \  \alpha_q + v_p \in \Delta_Q^\circ(q)\}.
\end{aligned}
\label{DiracStructCoord}
\end{equation}

\subsubsection{The Tulczyjew triple}

The Tulczyjew triple relates the spaces $T^*T^*Q$, $TT^*Q$, and $T^*TQ$ and helps bridge the gap between Lagrangian and Hamiltonian mechanics. These maps were first studied by Tulczyjew~\cite{Tu1977} in the context of a generalized Legendre transform. The first map is the usual flat map derived from the symplectic form $\Omega$ on $T^*Q$. We write $\Omega^\flat  : TT^*Q \to T^*T^*Q$ defined by
\begin{equation}
\Omega^\flat(v) \cdot u = \Omega(v,u).
\end{equation}
In coordinates,
\begin{equation}
\Omega^\flat(v) = (-v_p, v_q) \in T^*T^*Q.
\end{equation}
The second map, $\kappa_Q : TT^*Q \to T^*TQ$ is given locally by a permutation,
\begin{equation}
\kappa_Q : (q,p,\delta q, \delta p) \mapsto (q,\delta q, \delta p, p).
\end{equation}
A global definition of $\kappa_Q$ can be found in \cite{YoMa2006a}. A unique diffeomorphism $\kappa_Q$ exists for any manifold $Q$ \cite{YoMa2006a}. The third map, $\gamma_Q : T^*TQ \to T^*T^*Q$ is defined in terms of the first two,
\begin{equation}
\gamma_Q :=  \Omega^\flat \circ \kappa_Q^{-1}.
\end{equation}

\subsubsection{Lagrange--Dirac dynamical systems}

We are now equipped to define a Lagrange--Dirac dynamical system. Let $L: TQ \to \mathbb{R}$ be a given, possibly degenerate, Lagrangian. We define the Dirac differential of $L$ to be
\begin{equation}
\mathfrak{D}L(q,v) := \gamma_Q \circ \textbf{d} : TQ \to T^*T^*Q.
\end{equation}
Here $\textbf{d}$ denotes the usual exterior derivative operator so that $\textbf{d}L : TTQ \to T^*TQ$. For a curve $(q(t),v(t),p(t))\in TQ\oplus T^*Q$, we define $X_D$ to be the following partial vector field
\begin{equation}
X_D(q(t),v(t),p(t)) = (q(t),p(t),\dot{q}(t),\dot{p}(t)) \in TT^*Q.
\end{equation}
Then, the equations of motion for a Lagrange--Dirac dynamical system with Lagrangian $L$ and constraint distribution $\Delta_Q$ are given by
\begin{equation}
(X_D(q(t),v(t),p(t)), \mathfrak{D}L(q(t),v(t))) \in D_{\Delta_Q}(q(t),p(t)).
\label{Dirac_eqns_structure}
\end{equation}
In local coordinates, $\textbf{d}L(q,v) = (q,v, \pder{L}{q},\pder{L}{v})$ and
\begin{equation}
\gamma_Q : (q,\delta q, \delta p, p) \mapsto (q,p,-\delta p, \delta q),
\end{equation}
so we have
\begin{equation}
\mathfrak{D}L(q,v) = (q,\pder{L}{v},-\pder{L}{q},v).
\end{equation}
Then, using the coordinate expressions from \eqref{DiracStructCoord}, the equations determined by \eqref{Dirac_eqns_structure} are
\begin{equation}
\dot{q} = v \in \Delta_Q(q), \ \dot{p} - \pder{L}{q} \in \Delta_Q^\circ(q), \ p = \pder{L}{v}.
\label{UnforcedDiracCoords}
\end{equation}
The last equation comes from matching the basepoints of $X_D(q,p)$ and $\mathfrak{D}L(q,v)$. This is a set of differential algebraic equations on $TQ\oplus T^*Q$ whereas the Euler--Lagrange equations give an ODE system on $TQ$. We see that the first and last equations explicitly enforce the second-order curve condition and the Legendre transform, respectively. The middle equation reduces to the Euler--Lagrange equations in the absence of constraints. With constraints, the Euler--Lagrange relationship holds along the permissible directions. Explicit enforcement of the Legendre transform serves to enforce any primary constraints on the system.

\subsubsection{The Hamilton--Pontryagin principle}

Rather than the usual Hamilton's principle for curves on $TQ$, we apply the Hamilton--Pontryagin principle for curves on $TQ\oplus T^*Q$. This automatically incorporates a constraint distribution $\Delta_Q$ and any primary constraints coming from a degenerate Lagrangian. We have
\begin{equation}
\delta \int_0^T L(q(t),v(t)) - \langle p(t), \dot{q}(t) - v(t)\rangle \ dt = 0,
\label{HPprinciple}
\end{equation}
for variations $\delta q \in \Delta_Q(q)$ with fixed endpoints and arbitrary variations $\delta v, \delta p$ together with the constraint $\dot{q} \in \Delta_Q(q)$. This principle yields precisely the Lagrange--Dirac equations of motion \eqref{UnforcedDiracCoords}.

\subsubsection{The Lagrange--d'Alembert--Pontryagin principle and Lagrange--Dirac systems with external forces}

Suppose we have an external force field $F: TQ \to T^*Q$ acting on the system. As in the classical Lagrangian case \cite{MaRa1999}, we take the horizontal lift of $F$ to define $\tilde{F} : TQ \to T^*T^*Q$ by
\begin{equation}
\langle \tilde{F}(q,v), w\rangle = \langle F(q,v), T\pi_Q(w)\rangle.
\end{equation}
In local coordinates, $\tilde{F}(q,v) = (q,p,F(q,v),0)$. The equations of motion for the forced system are given by
\begin{equation}
(X_D(q,v,p), \mathfrak{D}L(q,v) - \tilde{F}(q,v)) \in D_{\Delta_Q}(q,p).
\end{equation}
As before, we can derive the local coordinate equations from this, producing
\begin{equation}
\dot{q} = v \in \Delta_Q(q), \ \dot{p} - \pder{L}{q} - F \in \Delta_Q^\circ(q), \ p = \pder{L}{v}.
\label{ForcedDiracCoords}
\end{equation}
So, only the second equation changes when forces are introduced. Equations \eqref{ForcedDiracCoords} reduce to the usual forced Euler--Lagrange equations in the absence of constraints.

We must also incorporate the work of the forces into the variational principle. This is done in exactly the same way as forces are appended to Hamilton's principle in the usual forced Lagrangian setting \cite{MaWe2001}. In that setting, one obtains the Lagrange--d'Alembert principle. Here, we arrive at the \emph{Lagrange--d'Alembert--Pontryagin principle},
\begin{equation}
\delta \int_0^T L(q,v) + \langle p, \dot{q} - v\rangle \ dt + \int_0^T \langle F(q,v),\delta q\rangle \ dt = 0,
\label{LDAprinciple}
\end{equation}
for variations $\delta q \in \Delta_Q(q)$ with fixed endpoints and arbitrary variations $\delta v, \delta p$ together with the constraint $\dot{q} \in \Delta_Q(q)$. The addition of the forcing terms here again produces \eqref{ForcedDiracCoords}.

\subsection{Interconnection of Lagrange--Dirac systems}
In this section we review the interconnection of continuous Lagrange--Dirac systems laid out in \citet{JaYo2014}. Throughout this section we assume that we are connecting two systems $(L^1, \Delta_{Q_1})$ on $Q_1$ and $(L^2,\Delta_{Q_2})$ on $Q_2$. The results easily extend to the interconnection of a finite number of systems, as shown in \cite{JaYo2014}. The interconnected system will then evolve on $Q = Q_1\times Q_2$. The interconnection of the two systems has both a variational formulation and a formulation in terms of the interconnection of the two starting Dirac structures, $D_{\Delta_{Q_1}}$ and $D_{\Delta_{Q_2}}$. This interconnection of Dirac structures in turn involves the direct sum of $D_{\Delta_{Q_1}}$ and $D_{\Delta_{Q_2}}$, a product on Dirac structures, and an interaction Dirac structure $D_{\text{int}}$.

\subsubsection{Standard interaction Dirac structures}

Let $\Sigma_Q\subset TQ$ be a regular distribution on $Q$ describing the interaction between systems 1 and 2. Lift this distribution to $T^*Q$ to define
\begin{equation}
\Sigma_{\text{int}} = (T\pi_Q)^{-1}(\Sigma_Q) \subset TT^*Q.
\end{equation}
Then, the \emph{standard interaction Dirac structure} $D_{\text{int}}$ on $T^*Q$ is given by
\begin{equation}
D_{\text{int}}(q,p) = \Sigma_{\text{int}}(q,p)\oplus \Sigma_{\text{int}}^\circ(q,p),
\end{equation}
for $\Sigma_{\text{int}}^\circ$ the annihilator of $\Sigma_{\text{int}}$.

As mentioned above, any Dirac structure on a manifold $M$ defines an associated distribution $\Delta_M \subset TM$ and a bilinear map $\omega_{\Delta_M} : \Delta_M \times \Delta_M \to \mathbb{R}$ that is well-defined on $\Delta_M$. Taking $D = \Delta \oplus \Delta^\circ$ produces $\Delta_M = \Delta$ and $\omega_{\Delta_M} \equiv 0$. Thus, the distribution associated with $D_{\text{int}}$ is $\Sigma_{\text{int}}$, and the associated two-form is the zero form. The zero form obviously extends to the whole of $T^*Q$, so $D_{\text{int}}$ can equivalently be generated from $\Sigma_{\text{int}}$ and $\omega \equiv 0$.

\subsubsection{The direct sum of Dirac structures}

Given two Dirac structures $D_1$ and $D_2$ on $M_1$ and $M_2$, the direct sum $D_1\oplus D_2$ is the vector bundle over $M_1\times M_2$ given by
\begin{multline}
D_1\oplus D_2 (x_1,x_2) = 
 \{((v_1,v_2),(\alpha_1,\alpha_2)) \in T_{(x_1,x_2)}(M_1\times M_2) \oplus T_{(x_1,x_2)}^*(M_1\times M_2) \ | \\ 
 (v_1,\alpha_1) \in D_1(x_1) \ \text{and} \ (v_2,\alpha_2) \in D_2(x_2)\}.
\end{multline}
From \cite{JaYo2014}, we have that $D_1\oplus D_2$ is itself a Dirac structure over $M_1\times M_2$. In the particular case of induced Dirac structures, it was shown in \cite{JaYo2014} that $D_{\Delta_{Q_1}}\oplus D_{\Delta_{Q_2}} = D_{\Delta_{Q_1}\oplus \Delta_{Q_2}}$.

\subsubsection{The tensor product of Dirac structures}

The interconnection of Dirac structures relies on a product operation on Dirac structures referred to as the \emph{Dirac tensor product}. We have the following characterization of the Dirac tensor product.
\begin{defn}[\citet{JaYo2014}] Let $D_a$ and $D_b$ be Dirac structures on $M$. We define the Dirac tensor product
\begin{equation}
\begin{aligned}
D_a\boxtimes D_b &= \{ (v,\alpha) \in TM\oplus T^*M \ | \ \exists \beta \in T^*M \\
&\qquad\text{such that} \ (v,\alpha+\beta) \in D_a , (v,-\beta) \in D_b\}.
\end{aligned}
\end{equation}
\end{defn}
An equivalent definition is given in \cite{Gu2011}. Let $D_{\Delta}$ be an induced Dirac structure on $Q$ and $D_{\text{int}}$ the standard interaction Dirac structure defined above. Then, $D_\Delta \boxtimes D_{\text{int}}$ is a Dirac structure when $\Delta \cap \Sigma_Q$ is a regular distribution \cite{JaYo2014}.

\subsubsection{Interconnection of Dirac structures}

Recall that we wish to connect the systems $(L^1, \Delta_{Q_1})$ and $(L^2, \Delta_{Q_2})$ with associated Dirac structures $D_{\Delta_{Q_1}}$ and $D_{\Delta_{Q_2}}$. The smooth distribution $\Sigma_Q\subset TQ$ describes their interaction and is used to define the interaction Dirac structure $D_{\text{int}} = \Sigma_{\text{int}} \oplus \Sigma_{\text{int}}^\circ$, where $\Sigma_{\text{int}} = (T\pi_Q)^{-1}(\Sigma_Q)\subset TT^*Q$. As before, $Q = Q_1\times Q_2$ will be the configuration manifold of the interconnected system.

Given two Dirac structures $D_a$ and $D_b$ on $Q_a$ and $Q_b$, respectively, and an interaction Dirac structure $D_{\text{int}}$ on $Q = Q_a\times Q_b$, the \emph{interconnection of $D_a$ and $D_b$ through $D_{\text{int}}$} is
\begin{equation}
(D_a\oplus D_b)\boxtimes D_{\text{int}}.
\end{equation}

We noted above that $D_{\Delta_{Q_1}} \oplus D_{\Delta_{Q_2}} = D_{\Delta_{Q_1}\oplus\Delta_{Q_2}}$. We have the following proposition for the interconnection of $D_{\Delta_{Q_1}}$ and $D_{\Delta_{Q_2}}$ through the standard interaction Dirac structure $D_{\text{int}} = \Sigma_{\text{int}}\oplus \Sigma_{\text{int}}^\circ$. %(Recall that $\Sigma_Q \subset TQ$ and $\Sigma_{\text{int}} = (T\pi_Q)^{-1}(\Sigma_Q) \subset TT^*Q$.)
\begin{prop}[\citet{JaYo2014}] If $\Delta_{Q_1}\oplus \Delta_{Q_2}$ and $\Sigma_Q$ intersect cleanly, i.e., $(\Delta_{Q_2}\oplus \Delta_{Q_2})\cap\Sigma_Q$ has locally constant rank, then the interconnection of $D_{\Delta_{Q_1}}$ and $D_{\Delta_{Q_2}}$ through $D_{\text{int}}$ is locally given by the Dirac structure induced from $(\Delta_{Q_2}\oplus \Delta_{Q_2})\cap\Sigma_Q$ as, for each $(q,p)\in T^*Q$,
\begin{multline}
(D_{\Delta_{Q_1}}\oplus D_{\Delta_{Q_2}})\boxtimes D_{\text{int}}(q,p) = \{ (v,\alpha)\in T_{(q,p)}T^*Q\times T_{(q,p)}^*T^*Q \ | \\
v\in \Delta_{T^*Q}(q,p) \ \text{and} \ \alpha - \Omega^\flat(q,p)\cdot v \in \Delta_{T^*Q}^\circ(q,p)\},
\end{multline}
where $\Delta_{T^*Q} = (T\pi_Q)^{-1}((\Delta_{Q_2}\oplus \Delta_{Q_2})\cap\Sigma_Q)$ and $\Omega = \Omega_1\oplus\Omega_2$, where $\Omega_1$ and $\Omega_2$ are the canonical symplectic structures on $T^*Q_1$ and $T^*Q_2$.
\end{prop}

Note that for $Q = Q_1\times Q_2$, the canonical symplectic form $\Omega_{T^*Q} = \Omega_{T^*Q_1}\oplus\Omega_{T^*Q_2}$. Thus, if we define 
\begin{equation}
\Delta_Q = (\Delta_{Q_1}\oplus \Delta_{Q_2})\cap\Sigma_Q,
\end{equation}
the previous proposition amounts to 
\begin{equation}
(D_{\Delta_{Q_1}}\oplus D_{\Delta_{Q_2}})\boxtimes D_{\text{int}} = D_{\Delta_Q}.
\end{equation}

\subsubsection{Interconnection of Lagrange--Dirac systems}

Set $L(q,v) = L^1(q_1,v_1) + L^2(q_2,v_2)$ and $\Delta_Q = (\Delta_{Q_1}\oplus \Delta_{Q_2})\cap\Sigma_Q$. Here, as usual, $(q,v)= (q_1,q_2,v_1,v_2) \in TQ = T(Q_1\times Q_2)$ in coordinates. Then, the interconnected system satisfies
\begin{equation}
(X_D(q,v,p), \mathfrak{D}L(q,v)) \in D_{\Delta_Q}(q,p).
\end{equation}
The interconnected system also satisfies the usual Hamilton--Pontryagin principle \eqref{HPprinciple} for $L$ and $\Delta_Q$.

Should there be any external forces $F_i: TQ_i \to T^*Q_i$ acting on the subsystems, those can be lifted to $Q$ by pullback with respect to $\pi_{Q_i}:Q\rightarrow Q_i$. That is to say that $F = \sum_i \pi_{Q_i}^*F_i$ represents the external forces acting on the interconnected system. Then, the total system solves the equations
\begin{equation}
(X_D(q,v,p), \mathfrak{D}L(q,v) - F) \in D_{\Delta_Q}(q,p),
\end{equation}
and satisfies the Lagrange--d'Alembert--Pontryagin principle \eqref{LDAprinciple}.

Note that in \cite{JaYo2014}, the forces considered in the interconnection process are interaction forces between subsystems, not external forces. As demonstrated in \cite{JaYo2014}, the constraints imposed by $\Sigma_Q$ have an equivalent representation in terms of internal interaction forces. We ignore the interaction force perspective for now, viewing interconnections as governed wholly by constraints $\Sigma_Q$. We will say more about bringing the interaction force perspective into discrete interconnections in the concluding sections.

\subsection{Discrete Dirac mechanics}

In this section, we review the discrete theory of Dirac mechanics and Dirac structures developed in \citet{LeOh2011}. We begin with a Lagrangian function $L:TQ\rightarrow\mathbb{R}$ and a continuous constraint distribution $\Delta_Q \subset TQ$.

\subsubsection{A discrete Tulczyjew triple}
Recall the continuous Tulczyjew triple, summarized in the following diagram.
\begin{equation}
\begin{tikzcd}
T^*TQ \arrow[bend left]{rrrr}{\gamma_Q} & & \arrow{ll}[swap]{\kappa_Q} TT^*Q \arrow{rr}{\Omega^\flat} & & T^*T^*Q .
\end{tikzcd}
\end{equation}
This is used to define the continuous Dirac differential $\mathfrak{D}L = (\gamma_Q \circ \textbf{d})L$.

In \cite{LeOh2011}, the authors define a discrete Tulczyjew triple using generating functions of a symplectic map $F: T^*Q\to T^*Q$. In coordinates, these are
\begin{align}
\kappa_Q^d : ((q_0,p_0),(q_1,p_1)) \mapsto (q_0,q_1,-p_0,p_1),\\
\Omega_{d+}^\flat : ((q_0,p_0),(q_1,p_1)) \mapsto (q_0,p_1,p_0,q_1),\\
\Omega_{d-}^\flat : ((q_0,p_0),(q_1,p_1)) \mapsto (p_0,q_1,-q_0,-p_1).
\end{align}
The distinction between $\Omega_{d\pm}^\flat$ comes from choosing either the Type II or Type III generating function in its definition, or equivalently, whether one chooses to endow $Q\times Q$ with a bundle structure over $Q$ by projecting onto the first or second component.

These maps define the (+) and ($-$) discrete Tulczyjew triples,
\begin{equation}
\begin{tikzcd}
T^*(Q\times Q) \arrow[bend left]{rrrr}{\gamma_Q^{d+}} & & \arrow{ll}[swap]{\kappa_Q} T^*Q\times T^*Q \arrow{rr}{\Omega_{d+}^\flat} & & T^*(Q\times Q^*),
\end{tikzcd}
\end{equation}
and
\begin{equation}
\begin{tikzcd}
T^*(Q\times Q) \arrow[bend left]{rrrr}{\gamma_Q^{d-}} & & \arrow{ll}[swap]{\kappa_Q} T^*Q\times T^*Q \arrow{rr}{\Omega_{d-}^\flat} & & T^*(Q^*\times Q).
\end{tikzcd}
\end{equation}
We use $\gamma_Q^{d\pm}$ to define a ($\pm$) discrete Dirac differential on $L_d$ and $\Omega_{d\pm}^\flat$ to define $(\pm)$ discrete induced Dirac structures.

\subsubsection{Discrete constraint distributions and discrete induced Dirac structures}
Recall that a continuous Lagrange--Dirac system on a manifold $Q$ has an associated constraint distribution $\Delta_Q \subset TQ$. With this we have a set of associated constraint one-forms $\{\omega^a\}$ such that
\begin{equation}
\Delta_Q^\circ(q) = \Span\{\omega^a(q)\}_{a=1}^m, \ \ \text{i.e.,} \ \ \Delta_Q = \cap_{a=1}^m \ker(\omega^a(q)).
\end{equation}
We define a discrete constraint distribution by discretizing these constraint one-forms. In the approach developed in \cite{LeOh2011}, we do this by using a retraction $R : TQ\to Q$, which is defined below.
\begin{definition}[{\citet[][Definition~4.1.1 on p.~55]{AbMaSe2008}}]
  A retraction on a manifold $Q$ is a smooth mapping $R: TQ \to Q$ with the following properties:
  Let $R_{q}: T_{q}Q \to Q$ be the restriction of $R$ to $T_{q}Q$ for an arbitrary $q \in Q$; then,
  \begin{enumerate}[(i)]
  \item $R_{q}(0_{q}) = q$, where $0_{q}$ denotes the zero element of $T_{q}Q$;
    \label{item:Retraction-i}
  \item with the identification $T_{0_{q}}T_{q}Q \simeq T_{q}Q$, $R_{q}$ satisfies
    \begin{equation*}
      \label{eq:TR}
      T_{0_{q}} R_{q} = \id_{T_{q}Q},
    \end{equation*}
    where $T_{0_{q}} R_{q}$ is the tangent map of $R_{q}$ at $0_{q} \in T_{q}Q$.
    \label{item:Retraction-ii}
  \end{enumerate}
\end{definition}
As with the Tulczyjew triple, we have a (+) and a ($-$) way of doing this, resulting in discrete forms $\omega_{d\pm}^a : Q\times Q \to \mathbb{R}$.
\begin{equation}
\omega_{d+}^a(q_0,q_1) = \omega^a(q_0)\left(R_{q_0}^{-1}(q_1)\right), \ \ \ \omega_{d-}^a(q_0,q_1) = \omega^a(q_1)\left(-R_{q_1}^{-1}(q_0)\right).
\end{equation}
The discrete constraint distribution is then defined as
\begin{equation}
\Delta_Q^{d\pm} = \{(q_0,q_1) \in Q\times Q \ | \ \omega_{d\pm}(q_0,q_1) = 0, \ a = 1,\dots,m\}.
\end{equation}
In the classical theory of variational integrators, the pair $(q_0,q_1)$ is thought of as the discrete analogue to a tangent vector in $TQ$. The $(\pm)$ formulation here can be thought of as a right and left formulation based on treating one of $q_0, q_1$ as the basepoint and the other as a representative of the velocity. Indeed, as noted in \cite{LeOh2011}, the distribution $\Delta_Q^{d+}$ constrains only $q_1$, while $\Delta_Q^{d-}$ constrains only $q_0$. This is consistent with what one would expect with nonholonomic constraints, where the velocities are constrained locally, but the positions are unconstrained.

Recall that a continuous Dirac structure on $T^*Q$ relies on the distribution $\Delta_{T^*Q} = (T\pi_Q)^{-1}(\Delta_Q) \subset TT^*Q$ for the canonical projection $\pi_Q:T^* Q\rightarrow Q$. At the discrete level, we define
\begin{equation}
\begin{aligned}
\Delta_{T^*Q}^{d+} &= (\pi_Q\times\pi_Q)^{-1}(\Delta_Q^{d\pm})\\
&= \left\{\left((q_0,p_0),(q_1,p_1)\right)\in T^*Q\times T^*Q \ | \ (q_0,q_1) \in \Delta_Q^{d\pm}\right\},
\end{aligned}
\end{equation}
and
\begin{equation}
\begin{aligned}
\Delta_{Q\times Q^*}^\circ &= \left\{(q,p,\alpha_q,0)\in T^*(Q\times Q^*) \ | \ \alpha_q \text{d}q\in \Delta_Q^\circ(q)\right\},\\
\Delta_{Q^*\times Q}^\circ &= \left\{(q,p,0,\alpha_q)\in T^*(Q^*\times Q) \ | \ \alpha_q \text{d}q\in \Delta_Q^\circ(q)\right\}.
\end{aligned}
\end{equation}
The distributions $\Delta_{T^*Q}^{d\pm}$ serve as the discrete analogues of $\Delta_{T^*Q}$, while $\Delta_{Q\times Q^*}^\circ$ and $\Delta_{Q^*\times Q}^\circ$ are the (+) and ($-$) discrete analogues of $\Delta_{T^*Q}^\circ$, respectively.

We then define discrete induced Dirac structures using these discrete distributions and the discrete maps $\Omega_{d\pm}^\flat$ defined earlier. We have
\begin{equation}
\begin{aligned}
D_{\Delta_Q}^{d+} &= \{\left((z,z^+),\alpha_{\hat{z}}\right) \in (T^*Q\times T^*Q)\times T^*(Q\times Q^*) \ | \\
&\qquad (z,z^+)\in \Delta_{T^*Q}^{d+}, \ \alpha_{\hat{z}} - \Omega_{d+}^\flat(z,z^+)\in \Delta_{Q\times Q^*}\}
\end{aligned}
\end{equation}
and
\begin{equation}
\begin{aligned}
D_{\Delta_Q}^{d-} &= \{\left((z^-,z),\alpha_{\tilde{z}}\right) \in (T^*Q\times T^*Q)\times T^*(Q^*\times Q) \ | \\
&\qquad (z^-,z)\in \Delta_{T^*Q}^{d-}, \ \alpha_{\tilde{z}} - \Omega_{d-}^\flat(z^-,z)\in \Delta_{Q^*\times Q}\}.
\end{aligned}
\end{equation}
Given $z = (q,p)$ and $z^+ = (q^+,p^+)$, then $\hat{z} = (q,p^+)$. Given $z^- = (q^-,p^-)$ and $z = (q,p)$, then $\tilde{z} = (p^-,q)$.

\subsubsection{The discrete Dirac differential and discrete Dirac mechanics}
We have two versions of the discrete Dirac differential,
\begin{equation}
\mathfrak{D}^+L_d = \gamma_Q^{d+}\circ \textbf{d}L_d \ \ \  \text{and} \ \ \ \mathfrak{D}^-L_d = \gamma_Q^{d-}\circ \textbf{d}L_d.
\end{equation}
Using the discrete vector field
\begin{equation}
X_d^k = \left((q_k,p_k),(q_{k+1},p_{k+1})\right) \in T^*Q\times T^*Q,
\end{equation}
we  have the following systems. A \emph{(+) discrete Lagrange--Dirac system} satisfies
\begin{equation}
(X_d^k, \mathfrak{D}^+L_d(q_k,q_k^+))\in D_{\Delta_Q}^{d+}.
\label{pDiscreteDiracStructFormulation}
\end{equation}
A \emph{($-$) discrete Lagrange--Dirac system} satisfies
\begin{equation}
(X_d^k,\mathfrak{D}^-L_d(q_{k+1}^-,q_{k+1}))\in D_{\Delta_Q}^{d-}.
\label{mDiscreteDiracStructFormulation}
\end{equation}
The variables $q_k^+$ and $q_{k+1}^-$ are the discrete analogues of the velocity variable. In coordinates, equation \eqref{pDiscreteDiracStructFormulation} produces the \emph{(+) discrete Lagrange--Dirac equations of motion},
\begin{subequations}
\begin{align}
0 &= \omega_{d+}^a(q_k, q_{k+1}) \quad a=1,\dots,m, \\
q_{k+1} &= q_k^+, \\
p_{k+1} &= D_2L_d(q_k,q_k^+), \\ 
p_k  &= - D_1L_d(q_k,q_k^+) + \mu_a\omega^a(q_k),
\end{align}\label{pDiracEqns}\end{subequations}
where $\mu_a$ are Lagrange multipliers, and the last equation uses the Einstein summation convention. Equation \eqref{mDiscreteDiracStructFormulation} produces the \emph{($-$) discrete Lagrange--Dirac equations of motion},
\begin{subequations}
\begin{align}
0&=\omega_{d-}^a(q_k, q_{k+1}) \quad a=1,\dots,m,\\ 
q_{k} &= q_{k+1}^- ,\\
p_{k} &= -D_1L_d(q_{k+1}^-,q_{k+1}), \\
p_{k+1}&= D_2L_d(q_{k+1}^-,q_{k+1}) + \mu_a\omega^a(q_{k+1}).
\end{align}\label{mDiracEqns}\end{subequations}
Again, $\mu_a$ are Lagrange multipliers, and the last equation makes use of the Einstein summation convention. Later, we will write these equations with the $q_k^+$ and $q_{k+1}^-$ variables eliminated for simplicity.

By eliminating the momentum variables, both $(\pm)$ equations simplify to the DEL equations in the unconstrained case, and they recover the nonholonomic integrators of \citet{CoMa2001}.

\subsubsection{Variational discrete Dirac mechanics}
The (+) discrete Hamilton--Pontryagin principle is
\begin{equation}
\delta\sum_{k=0}^{N-1}[L_d(q_k,q_k^+)+p_{k+1}(q_{k+1}-q_k^+)] = 0,
\label{pDiscreteHP}
\end{equation}
with variations that vanish at the endpoints, i.e. $\delta q_0 = \delta q_N = 0$, and the discrete constraints $(q_k,q_{k+1}) \in \Delta_Q^{d+}$. We also impose the constraint $\delta q_k \in \Delta_Q(q_k)$ after computing variations inside the sum. The variable $q_k^+$ serves as the discrete analog to the introduction of $v$ in the continuous principle.

The ($-$) discrete Hamilton--Pontryagin principle is
\begin{equation}
\delta\sum_{k=0}^{N-1} [L_d(q_{k+1}^-,q_{k+1}) - p_k(q_k-q_{k+1}^-)] = 0 \ .
\label{mDiscreteHP}
\end{equation}
The variable $q_k^-$ now plays the role of the discrete velocity. Again we take variations that vanish at the endpoints and impose the constraint $\delta q_k \in \Delta_Q(q_k)$. We now impose the discrete constraints $(q_k,q_{k+1})\in \Delta_Q^{d-}$.

As shown in \cite{LeOh2011}, computing variations of \eqref{pDiscreteHP} yields \eqref{pDiracEqns}, and computing variations for \eqref{mDiscreteHP} yields \eqref{mDiracEqns}. Thus, in direct analogy with the continuous case, we have equivalent variational and Dirac structure formulations of discrete Lagrange--Dirac mechanics.

\section{(+) vs. ($-$) Discrete Dirac mechanics}
Before getting to the interconnected systems results, we say a few words about the distinction between the (+) and ($-$) formulations of discrete Dirac mechanics laid out in \cite{LeOh2011}. Later sections will focus on interconnections of (+) discrete Dirac systems as that turns out to be the proper formulation for simulating forward in time.

In their full form, the (+) discrete Dirac equations are only generally solvable for forward time integration (moving forward in index), and the ($-$) discrete Dirac equations are only generally solvable for backward time integration (moving backward in index). This follows from the implicit function theorem. It also mirrors the case of the augmented approach to holonomic constraints laid out in \cite{MaWe2001}, which has a similar form. 

In momentum-matched form, the discrete Dirac equations become
\begin{align*}
D_2L_d(q_{k-1},q_k)+D_1L_d(q_k,q_{k+1})+\mu_a\omega^a(q_k) &= 0, & k&=1,\dots,N-1,\\
\omega_{d\pm}^a(q_k,q_{k+1}) &= 0, & k&=0,\dots,N-1.
\end{align*}
So the only distinction between the position trajectories of (+) and ($-$) is, potentially, in the way the constraints are discretized. The two methods generate the same trajectory when
\begin{align}
\omega_{d+}^a(q_k,q_{k+1}) = 0 &\iff \omega_{d-}^a(q_k,q_{k+1}) = 0.
\end{align}
For the retraction-based definition of $\omega_{d+}^a$ in \cite{LeOh2011}, this requires
\begin{align}
\omega^a(q_k)\cdot R_{q_k}^{-1}(q_{k+1}) = 0 &\iff \omega^a(q_{k+1})\cdot -R_{q_{k+1}}^{-1}(q_k) = 0.
\end{align}
This holds, for instance, for a force that is independent of the base point, and a retraction whose inverse is antisymmetric.
%a base point-independent force and an antisymmetric inverse retraction. 
For example, an equality constraint between two redundant variables will be independent of the base point, and the vector space retraction $R_q(v) = q+hv$ has an inverse that is antisymmetric in $(q_k, q_{k+1})$.
\begin{equation}
R_{q_k}^{-1}(q_{k+1}) = (q_{k+1}-q_k)/h = -R_{q_{k+1}}^{-1}(q_k).
\end{equation}
If we consider more general discretizations for $\omega_{d\pm}^a$, we could purposefully choose symmetric discretizations so that the (+) and ($-$) formulations generate the same position trajectories.

\section{Discrete Dirac interconnections}

In this section we present results for interconnecting a finite number of systems on $Q_1,\dots,Q_n$ to form a system on $Q = Q_1\times\cdots\times Q_n$.  Here and throughout the section, let $\pi_{Q_i}$ denote the projection from $Q$ onto $Q_i$ and $T\pi_{Q_i}$ denote the tangent lift of $\pi_{Q_i}$. In coordinates, we have $q = (q_1,\dots,q_n), v_q = (q_1,\dots,q_n,v_{q_1},\dots,v_{q_n})$ with $\pi_{Q_i}(q) = q_i$ and $T\pi_{Q_i}(v_q) = (q_i,v_{q_i})$. At the continuous level, we have two equivalent views of Dirac interconnections: through variational principles and through Dirac structures \cite{JaYo2014}. We always have an interconnection distribution $\Sigma_Q \subset TQ$ describing the interaction between the two systems. We can think of the interconnected system as the system generated variationally by $L(q,\dot{q}) = L^1(T\pi_{Q_1}(q,\dot{q})) + L^2(T\pi_{Q_2}(q,\dot{q}))$ and $\Delta_Q = (\Delta_{Q_1}\oplus \Delta_{Q_2})\cap\Sigma_Q$. To view interconnection in terms of Dirac structures, we write $(X_D, d_DL(q,v)) \in (D_{\Delta_{Q_1}}\oplus D_{\Delta_{Q_2}})\boxtimes D_{\text{int}}$ for the same Lagrangian. Here $D_{\text{int}}$ is a Dirac structure on $T^*Q$ derived from $\Sigma_Q$ and $\boxtimes$ is the Dirac tensor product defined earlier.

These two views of interconnection are completely equivalent, so that, in particular, $(D_{\Delta_{Q_1}}\oplus D_{\Delta_{Q_2}})\boxtimes D_{\text{int}} = D_{\Delta_Q}$ for $\Delta_Q = (\Delta_{Q_1}\oplus \Delta_{Q_2})\cap\Sigma_Q$. We mimic each viewpoint at the discrete level, producing an analogous equivalence between the two approaches.

\subsection{Interconnecting two discrete Dirac systems variationally through $\Sigma_Q$}
Suppose we have two systems $(L^1, \Delta_{Q_1})$ and $(L^2, \Delta_{Q_2})$ with configuration manifolds $Q_1$ and $Q_2$. Suppose we also have a distribution $\Sigma_Q \subset TQ$ for $Q = Q_1\times Q_2$ describing the interconnection of systems 1 and 2. Then, from \cite{JaYo2014}, we know that the interconnected system is again a Dirac system with Lagrangian $L(q,v) = L^1(q_1,v_1)+L^2(q_2,v_2)$ and distribution $\Delta_Q = (\Delta_{Q_1}\oplus\Delta_{Q_2})\cap\Sigma_Q$. To discretize any of these systems in the way laid out in \cite{LeOh2011}, we must choose a discretization scheme $L\mapsto L_d$ and a retraction $R: TQ\to Q$. We will assume our discretization scheme is linear in $L$, i.e. for $L = L^1(T\pi_{Q_1}(q,v))+L^2(T\pi_{Q_2}(q,v))$ we get $L_d(q_k,q_{k+1}) = L_d^1(q_k^1,q_{k+1}^1)+L_d^2(q_k^2,q_{k+1}^2)$. This is a relatively weak assumption. Schemes for constructing $L_d$ are based on approximating the exact discrete Lagrangian given by
\begin{equation}
L_d^E(q_0,q_1;h) = \int_0^h L(q(t),\dot{q}(t))dt,
\end{equation}
where $q(t)$ satisfies the appropriate differential equations (Euler-Lagrange, forced Euler-Lagrange, Dirac, etc.) and the boundary conditions $q(0)=q_0$, $q(h)=q_1$. Any forces or constraints are discretized separately, though there is an argument to be made in favor of using the same discretization scheme for each \cite{helen_thesis}. Since the exact discrete Lagrangian is linear in $L$, discretizations are most often linear as well. For instance, discretization based on applying quadrature to the integral in $L_d^E$ will satisfy linearity in $L$.

\subsubsection{Compatible constraint discretizations}
The relevant constraint distributions in interconnection are $\Delta_{Q_1},\dots,\Delta_{Q_n}$, $\Sigma_{Q}$ and $\Delta_Q = (\Delta_{Q_1} \oplus \cdots \oplus \Delta_{Q_n})\cap\Sigma_{Q}$. To get equivalence between interconnecting systems before and after discretization, we need to make a particular choice of basis for $\Delta_Q^\circ$ and assume a \emph{compatible} constraint discretization, defined below. We will address the sufficient conditions on $\Delta_{Q_1},\dots,\Delta_{Q_n}$, and $\Sigma_{Q}$ to ensure that the resulting discrete equations of motion have an admissible solution in future work. But, at the minimum, this will depend on the extent to which the individual nonholonomic constraint distributions $\Delta_{Q_i}$ are compatible with the interconnection constraint $\Sigma_{Q}$ projected onto the corresponding $Q_i$. 

From $\Delta_Q = (\Delta_{Q_1} \oplus \cdots \oplus \Delta_{Q_n})\cap\Sigma_{Q}$, we have $\Delta_Q^\circ = (\Delta_{Q_1} \oplus \cdots \oplus \Delta_{Q_n})^\circ\cup\Sigma_{Q}^\circ$. Thus, we can construct a basis for $\Delta_Q^\circ$ from the bases of $\Delta_{Q_i}^\circ$ and $\Sigma_{Q}^\circ$. Let $\{\omega_i^a(q^i)\}_{a=1}^{m_i}$ denote a basis for $\Delta_{Q_i}^\circ(q^i)$. We construct a basis for $(\Delta_{Q_1}\oplus\cdots\oplus\Delta_{Q_n})^\circ(q)$ from the individual bases $\{\omega_i^a(q^i)\}_{a=1}^{m_i}$. Let $\pi_{Q_i}(q)\in Q_i$ denote the $i^{th}$ component projection of $q\in Q$. Define $\tilde{\omega}_{i}^a(q)$ by $\tilde{\omega}_i^a(q)\cdot v_{q} = \omega_i^a(\pi_{Q_i}(q))\cdot T\pi_{Q_i}(v_{q}).$ Then, $\{\{\tilde{\omega}_{i}^a(q)\}_{a=1}^{m_i}\}_{i=1}^n$ forms a basis for $(\Delta_{Q_1}\oplus\cdots\oplus\Delta_{Q_n})^\circ(q)$. Select a basis for $\Sigma_Q^\circ(q) = \text{span}\{\alpha^b(q)\}_{b=1}^l$. Then, $\Delta_Q^\circ(q) = \text{span}\{\tilde{\omega}_i^a(q), \alpha^b(q)\}$, with the appropriate ranging of indices. This will always be our chosen basis for $\Delta_Q^\circ$. 

We will call a constraint discretization \emph{compatible} if
\begin{equation}
\tilde{\omega}_{d+,i}^a(q_k,q_{k+1}) = \omega_{d+,i}^a(q_k^i,q_{k+1}^i).
\end{equation}
We use the notation $q_k^i$ to mean the $i^{th}$ component at the $k^{th}$ time-step. So the full coordinate expression at $t_k$ is $q_k = (q_k^1,\dots,q_k^n)$ with $q_k^i \in Q_i$. For retraction-based discretizations, we make use of the following lemma.
\begin{lem}
For $R_1,\dots,R_n$ retractions on $Q_1,\dots,Q_n$, respectively, $R_1\times\cdots\times R_n$ is a retraction on $Q = Q_1\times\cdots\times Q_n$.
\end{lem}
Compatibility of retraction-based discretizations requires the use of $R_1\times\cdots\times R_n$ as the retraction on $Q = Q_1\times\cdots\times Q_n$.

\subsubsection{Discrete interconnections using compatible constraint discretizations}

Discretizing individual systems before interconnection yields
\begin{subequations}
\begin{align}
p_{k+1}^i &= D_2L_d^i(q_k^i,q_{k+1}^i),\\
p_k^i &= -D_1L_d^i(q_k^i, q_{k+1}^i) + \mu_a\omega_{i}^a(q_k^i),\\
0&=\omega_{d+,i}^a(q_k^i,q_{k+1}^i), & a&=1,\dots,m_i.
\end{align}
\end{subequations}
Here, $L_d^i$ have been discretized according to some scheme linear in $L$, and 
\begin{equation}
\omega_{d+,i}^a(q_k^i,q_{k+1}^i) = \omega_i^a(q_k^i)(R_{i,q_k^i}^{-1}(q_{k+1}^i)).
\end{equation}

As above, take $\Sigma_Q^\circ(q) = \text{span}\{\alpha^b(q)\}_{b=1}^l$, so each $\alpha^b(q) \in T_q^*Q$. Define $\alpha_i^b(q)\in T^*Q_i$ by $\alpha_i^b(q)\cdot v_{q_i} = \alpha^b(q)\cdot v_{q_i}^h$ for $v_{q_i}^h$ the horizontal lift of $v_{q_i}$. To interconnect the discrete systems above, we need to append an $\alpha_i^b(q)$ term, which represents an unknown force of constraint, to each equation for $p_k^i$ and impose the $\alpha_{d+}^b$ constraints to $(q_k,q_{k+1})\in Q$. That is, the interconnected system is
\begin{subequations}
\begin{align}
p_{k+1}^i &= D_2L_d^i(q_k^i,q_{k+1}^i), \label{int_eqns_1}\\
p_k^i &= -D_1L_d^i(q_k^i, q_{k+1}^i) + \mu_a\omega_{i}^a(q_k^i)+\lambda_b\alpha_i^b(q_k), \label{int_eqns_2}\\
0&=\omega_{d+,i}^a(q_k^i,q_{k+1}^i), & a&=1,\dots,m_i, \label{int_eqns_3} \\
0&=\alpha_{d+}^b(q_k,q_{k+1}), & b &= 1,\dots,l. \label{int_eqns_end}
\end{align}
\end{subequations}
Note that all of the $\alpha$ terms depend on the entire coordinate $q_k = (q_k^1,\dots,q_k^n)$, not just on the $i^{th}$ component $q_k^i$.

\begin{thm}
Assume $\phi: L \to L_d$ is linear and the constraint discretization is compatible. Then, the discretely interconnected equations \eqref{int_eqns_1}--\eqref{int_eqns_end} are equivalent to the (+) discrete Dirac equations for $(L,\Delta_Q) = (L^1+\cdots+L^n, (\Delta_{Q_1}\oplus\cdots\oplus\Delta_{Q_n})\cap\Sigma_Q)$.
\end{thm}
\begin{proof}
We just have to consider the discretization of $(L, \Delta_Q)$ component-wise. The discretization of the monolithic system using $\phi$ and $R$ yields the usual (+) discrete Dirac equations,
\begin{subequations}
\begin{align}
p_{k+1} &= D_2 L_d(q_k, q_{k+1}),\\
p_k &= -D_1L_d(q_k, q_{k+1}) + \eta_c\beta^c(q_k), \label{pk_int_proof}\\
0&=\beta_{d+}^c(q_k,q_{k+1}) = 0, \quad c = 1, \dots, m. \label{int_proof_constraints}
\end{align}
\end{subequations}
Here $m$ is the dimension of $\Delta_Q^\circ$. From our assumptions on the linearity of $\phi$, the first equation decomposes component-wise to give equation \eqref{int_eqns_1}.

In this notation, $\{\beta^c(q_k)\}_c$ is a basis for $\Delta_Q^\circ(q_k) = ((\Delta_{Q_1}\oplus\cdots\oplus\Delta_{Q_n})\cap\Sigma_Q)^\circ$. As we will see, choosing an appropriate basis leads to equations \eqref{int_eqns_2}--\eqref{int_eqns_end}. As above, define $\beta_i^c(q)\in T^*Q_i$ by $\beta_i^c(q)\cdot v_{q_i} = \beta^c(q)\cdot v_{q_i}^h$ for $v_{q_i}^h$ the horizontal lift of $v_{q_i}$. Then, equation \eqref{pk_int_proof} decomposes into
\begin{equation}
p_k^i = -D_1L_d^i(q_k^i,q_{k+1}^i)+\eta_c\beta_i^c(q_k). \label{pk_int_proof_split}
\end{equation}

Taking the basis defined above, we have $\{\beta^c\} = \{\omega_i^a(q), \alpha^b(q)\}$, so  equation \eqref{pk_int_proof_split} becomes equation \eqref{int_eqns_2}. Assume we construct $\beta_{d+}^c$ via a compatible discretization. Then, equation \eqref{int_proof_constraints} accounts for equations \eqref{int_eqns_3} and \eqref{int_eqns_end}.
\end{proof}

Thus, given a finite number of Lagrange--Dirac systems $(L^i,\Delta_{Q_i})$ together with the interconnection constraint $\Sigma_Q$, we have shown how to interconnect the discrete systems generated by $(L_d^i,\Delta_{Q_i}^{d+},\Delta_{Q_i}^\circ)$ through $\Sigma_Q^{d+}$ and $\Sigma_Q^\circ$ to obtain the discretization of the fully interconnected system.

\subsection{Discrete interconnections as a product on discrete Dirac structures}

To mimic the continuous case, we would like to say that this discrete interconnection process corresponds to a discrete Dirac tensor product on discrete Dirac structures. That is, we would like for the discretization of the interconnected system, which can be expressed as $(X_d^k,\mathfrak{D}^+L_d(q_k,q_k^+))\in D_{\Delta_Q}^{d+}$, to be equivalently expressed as  $(X_d^k,\mathfrak{D}^+L_d(q_k,q_k^+))\in (D_{\Delta_{Q_1}}^{d+} \oplus D_{\Delta_{Q_2}}^{d+})\boxtimes_d D_{\text{int}}^{d+}$ for $D_{\text{int}}^{d+}$ defined from $\Sigma_Q$, some definition of $\boxtimes_d$, and the appropriate notion of $\oplus$.

\subsubsection{The direct sum of induced discrete Dirac structures}

The definition of $\oplus$ for induced discrete Dirac structures is relatively obvious. We make it precise in this section to ensure that the convenient properties of using $\oplus$ on induced Dirac structures carry over to the discrete setting. Suppose, again, that $Q = Q_1\times Q_2$ and that we have two constraint distributions $\Delta_{Q_1}\subset TQ_1$ and $\Delta_{Q_2}\subset TQ_2$. We can derive each distribution from its annihilator as $\Delta_{Q_i}(q_i) = \cap_a \ker(\omega_i^a(q_i))$ for $\{\omega_i^a(q_i)\}_a$ a basis for $\Delta_{Q_i}^\circ(q_i)$. The direct sum distribution on $Q$ has annihilator given by $(\Delta_{Q_1}\oplus \Delta_{Q_2})^\circ = \Delta_{Q_1}^\circ \oplus \Delta_{Q_2}^\circ$, so we can construct a basis for it by extending the bases of $\Delta_{Q_i}^\circ$. As in the last section, we use $\pi_{Q_i}: Q \to Q_i$ to denote component projections from $Q$. To extend $\omega_i^a$, we denote by $\tilde{\omega}_i^a$ the one-form on $Q$ such that $\tilde{\omega}_i^a(q)\cdot v_q = \omega_i^a(\pi_{Q_i}(q))\cdot T\pi_{Q_i}(v_q)$. In coordinates $\tilde{\omega}_1^a = (\omega_1^a, 0)$ and $\tilde{\omega}_2^b = (0,\omega_2^b)$. Then, the distribution $\Delta_{Q_1}\oplus\Delta_{Q_2}$ has a local expression as $(\Delta_{Q_1}\oplus \Delta_{Q_2})(q) = [\cap_a \ker(\tilde{\omega}_1^a(q))]\cap[\cap_b \ker(\tilde{\omega}_2^b(q))]$.

The direct sum of continuous Dirac structures $D_{\Delta_{Q_1}}$ and $D_{\Delta_{Q_2}}$ is given by $D_{\Delta_{Q_1}}\oplus D_{\Delta_{Q_2}} = D_{\Delta_{Q_1}\oplus\Delta_{Q_2}}$. Fiber-wise, this is given by 
\begin{equation}
(D_{\Delta_{Q_1}}\oplus D_{\Delta_{Q_2}})(q,p) = D_{\Delta_{Q_1}}(T^*i_{Q_1}(q,p))\oplus D_{\Delta_{Q_2}}(T^*i_{Q_2}(q,p)),
\end{equation}
where $i_{Q_i} : Q_i \hookrightarrow Q$ is the inclusion and $T^*i_{Q_i}$ its tangent lift. In coordinates,
 \begin{multline}
 \{ (v,\alpha) = (v_1,v_2,\alpha_1,\alpha_2) \in T_{(q_1,q_2,p_1,p_2)}T^*Q \ | \ (v_1,\alpha_1) \in D_{\Delta_{Q_1}}(q_1,p_1) \\
 \text{and} \ (v_2,\alpha_2) \in D_{\Delta_{Q_2}}(q_2,p_2)\}.
 \end{multline}
 We mimic this coordinate expression at the discrete level with the following definition.

\begin{defn}
Given two discrete induced Dirac structures $D_{\Delta_{Q_1}}^{d+}\subset (T^*Q_1\times T^*Q_1)\times T^*(Q_1\times Q_1^*)$ and $D_{\Delta_{Q_2}}^{d+} \subset (T^*Q_2\times T^*Q_2)\times T^*(Q_2\times Q_2^*)$, define their direct sum $D_{\Delta_{Q_1}}^{d+}\oplus D_{\Delta_{Q_2}}^{d+} \subset (T^*Q\times T^*Q)\times T^*(Q\times Q^*)$ coordinate-wise as
\begin{multline}
D_{\Delta_{Q_1}}^{d+}\oplus D_{\Delta_{Q_2}}^{d+} = \{((z,z+),\alpha_{\hat{z}}) \ | \ ((q_1,p_1,q_1^+,p_1^+),(q_1,p_1^+,\alpha_{q_1},\alpha_{p_1}))\in D_{\Delta_{Q_1}}^{d+} \\ 
\text{and} \ ((q_2,p_2,q_2^+,p_2^+),(q_2,p_2^+,\alpha_{q_2},\alpha_{p_2}))\in D_{\Delta_{Q_2}}^{d+} \}
\end{multline}
Here, we have partitioned the coordinates as 
\begin{equation}
(z,z^+) = (q,p,q^+,p^+) = (q_1,q_2,p_1,p_2,q_1^+,q_2^+,p_1^+,p_2^+)
\end{equation}
and
\begin{equation}
\alpha_{\hat{z}} = (q,p^+,\alpha_q,\alpha_p) = (q_1,q_2,p_1^+,p_2^+,\alpha_{q_1},\alpha_{q_2},\alpha_{p_1},\alpha_{p_2}).
\end{equation}
\label{discrete_dirac_struct_sum_coord_def}
\end{defn}

We define the direct sum of two discrete constraint distributions as follows.

\begin{defn}
The direct sum of two discrete constraint distributions is given by
\begin{equation}
\begin{aligned}
\Delta_{Q_1}^{d+} \oplus \Delta_{Q_2}^{d+} &= \{(q_1,q_2,q_1^+,q_2^+) \in Q\times Q \  | \\ 
&\qquad (q_1,q_1^+) \in \Delta_{Q_1}^{d+}  \ \text{and} \ (q_2,q_2^+) \in \Delta_{Q_2}^{d+}\}.
\end{aligned}
\end{equation}
\end{defn}

We have the following useful lemma.

\begin{lem}
Assume we use the same separable discretization scheme to construct $\omega_{1,d+}^a$, $\omega_{2,d+}^b$, $\tilde{\omega}_{1,d+}^a,$ and $\tilde{\omega}_{2,d+}^b$. Then, $\Delta_{Q_1}^{d+} \oplus \Delta_{Q_2}^{d+} = (\Delta_{Q_1}\oplus \Delta_{Q_2})^{d+}$ and $D_{\Delta_{Q_1}}^{d+} \oplus D_{\Delta_{Q_2}}^{d+} = D_{\Delta_{Q_1}\oplus \Delta_{Q_2}}^{d+}$. Thus, $D_{\Delta_{Q_1}}^{d+}\oplus D_{\Delta_{Q_2}}^{d+}$ is again a discrete induced Dirac structure.
\end{lem}

\begin{proof}
We have
\begin{equation}
\begin{aligned}
(\Delta_{Q_1}\oplus \Delta_{Q_2})^{d+} &= \{(q,q^+) \in Q\times Q \ | \ \tilde{\omega}_{1,d+}^a(q,q^+) = 0 \ \text{and} \\ 
&\qquad \tilde{\omega}_{2,d+}^b(q,q^+) = 0 \ \text{for all} \ a,b\}
\end{aligned}
\end{equation}
and
\begin{align}
\Delta_{Q_i}^{d+} = \{(q_i,q_i^+) \in Q_i\times Q_i \ | \ \omega_{i,d+}^a(q_i,q_i^+) = 0 \ \text{for all} \ a\}.
\end{align}
By our assumptions, $\tilde{\omega}_{i,d+}^a(q,q^+) = \omega_{i,d+}^a(q_i,q_i^+)$. Thus, $\Delta_{Q_1}^{d+} \oplus \Delta_{Q_2}^{d+} = (\Delta_{Q_1}\oplus \Delta_{Q_2})^{d+}$.

To prove $D_{\Delta_{Q_1}}^{d+} \oplus D_{\Delta_{Q_2}}^{d+} = D_{\Delta_{Q_1}\oplus \Delta_{Q_2}}^{d+}$ we need to show that the conditions 
\begin{equation}
(q,q^+) \in (\Delta_{Q_1}\oplus \Delta_{Q_2})^{d+} \label{di_lem_dist_cond}
\end{equation}
and 
\begin{equation}
(q,p^+,\alpha_q-p,\alpha_p - q^+)\in \{(q,p,\beta,0) \ | \ \beta dq \in (\Delta_{Q_1}\oplus \Delta_{Q_2})^\circ(q)\} \label{di_lem_omega_cond}
\end{equation}
are equivalent to the conditions 
\begin{equation}
((q_1,p_1,q_1^+,p_1^+),(q_1,p_1^+,\alpha_{q_1},\alpha_{p_1}))\in D_{\Delta_{Q_1}}^{d+} \label{di_lem_separate1}
\end{equation}
and
 \begin{equation}
((q_2,p_2,q_2^+,p_2^+),(q_2,p_2^+,\alpha_{q_2},\alpha_{p_2}))\in D_{\Delta_{Q_2}}^{d+}. \label{di_lem_separate2}
\end{equation}
Using $\Delta_{Q_1}^{d+} \oplus \Delta_{Q_2}^{d+} = (\Delta_{Q_1}\oplus \Delta_{Q_2})^{d+}$, the distribution conditions implied by \eqref{di_lem_separate1} and \eqref{di_lem_separate2} are equivalent to \eqref{di_lem_dist_cond}. From \eqref{di_lem_separate1} and \eqref{di_lem_separate2} we also have
\begin{equation}
(q_1,p_1^+,\alpha_{q_1}-p_1,\alpha_{p_1} - q_1^+)\in \{(q_1,p_1,\beta,0) \ | \ \beta dq \in \Delta_{Q_1}^\circ\} 
\end{equation}
and
\begin{equation}
(q_2,p_2^+,\alpha_{q_2}-p_2,\alpha_{p_2} - q_2^+)\in \{(q_2,p_2,\beta,0) \ | \ \beta dq \in \Delta_{Q_2}^\circ\}.
\end{equation}
Thus, $(\alpha_{p_1},\alpha_{p_2}) - (q_1^+,q_2^+) = 0$. From $\alpha_{q_i} - p_i\in \Span\{\omega_i^a(q_i)\}$ we have $(\alpha_{q_1} - p_1,0) \in \Span\{\tilde{\omega}_1^a(q)\}$ and $(0, \alpha_{q_2} - p_2) \in \Span\{\tilde{\omega}_2^b(q)\}$. Thus, we have $(\alpha_{q_1} - p_1,\alpha_{q_2}-p_2) \in \Span\{\tilde{\omega}_1^a(q),\tilde{\omega}_2^b\}$, i.e. $\alpha_q - p \in (\Delta_{Q_1} \oplus \Delta_{Q_2})^\circ$. Hence, conditions  \eqref{di_lem_separate1} and \eqref{di_lem_separate2} also give \eqref{di_lem_omega_cond}.

Performing the same calculations from the reversed point of view, we can derive  \eqref{di_lem_separate1} and \eqref{di_lem_separate2} from \eqref{di_lem_dist_cond} and \eqref{di_lem_omega_cond}, proving that $D_{\Delta_{Q_1}}^{d+} \oplus D_{\Delta_{Q_2}}^{d+} = D_{\Delta_{Q_1}\oplus \Delta_{Q_2}}^{d+}$.
\end{proof}

For continuous distributions $\Delta_1$ and $\Delta_2$, a similar set of calculations show that
\begin{equation}
\Delta_1^{d+}\cap\Delta_2^{d+} = (\Delta_1\cap\Delta_2)^{d+}.
\end{equation}

\subsubsection{Defining $D_{\text{int}}^{d+}$ and $\boxtimes_d$}

We begin by defining $D_{\text{int}}^{d+}$. The distribution $\Sigma_Q$ defines the interconnection constraints on $Q = Q_1\times Q_2$. Lift $\Sigma_Q$ to $TT^*Q$, defining $\Sigma_{\text{int}} = (T\pi)^{-1}(\Sigma_Q)$. Then, the continuous interaction Dirac structure is induced by $\Sigma_{\text{int}}$ and $\Omega_{\text{int}}\equiv 0$. To discretize this construction, we define $\Omega_{\text{int}}^{d+}(q,p,q^+,p^+) = (q,p^+,0,0)$.
\begin{defn}
We defined the \emph{standard discrete interaction Dirac structure} to be
\begin{equation}
D_{\text{int}}^{d+} = \{ ((z,z^+),\alpha_{\hat{z}}) \ | \ (z,z^+) \in \Sigma_{\text{int}}^{d+}, \alpha_{\hat{z}} - \Omega_{\text{int}}^{d+}(z,z^+) \in \Sigma_{Q\times Q^*}^\circ\}(q)
\end{equation}
for $\Omega_{\text{int}}^{d+}(q,p,q^+,p^+) = (q,p^+,0,0)$.
\end{defn}
Here, as in the original definitions of the discrete Dirac structures, $z = (q,p), z^+ = (q^+,p^+), \hat{z} = (q,p^+)$. This discrete Dirac structure mirrors the induced discrete Dirac structure of \cite{LeOh2011} with $\Omega_{d\pm}^\flat$ replaced by $\Omega_{\text{int}}^{d+}$.

Recall, again, the continuous definition of $\boxtimes$.
\begin{defn}[\citet{JaYo2014}]
Let $D_a, D_b\in \text{Dir}(M)$, i.e., $D_a$ and $D_b$ are Dirac structures on $M$. We define the Dirac tensor product
\begin{equation}
\begin{aligned}
D_a\boxtimes D_b &= \{(v,\alpha)\in TM\oplus T^*M \ | \ \exists \beta\in T^*M \ \text{such that} \\ 
&\qquad (v,\alpha+\beta)\in D_a, (v,-\beta)\in D_b\}.
\end{aligned}
\end{equation}
\end{defn}
Mimicking this definition at the discrete level, we define $\boxtimes_d$ as follows.
\begin{defn}
Define the operation $\boxtimes_d$ on two discrete Dirac structures $D_1$ and $D_2$ by
\begin{equation}
\begin{aligned}
D_1\boxtimes D_2 &= \{((z,z^+),\alpha_{\hat{z}}) \ | \ \exists \beta_{\hat{z}} \in T_{\hat{z}}^*(Q\times Q^*) \\ 
&\qquad \text{with } ((z,z^+),\alpha_{\hat{z}}+\beta_{\hat{z}}) \in D_1,  ((z,z^+),-\beta_{\hat{z}})\in D_2\},
\end{aligned}
\end{equation}
where $\beta_{\hat{z}} = (q,p^+,\beta_q,\beta_{p^+}), \alpha_{\hat{z}} = (q,p^+,\alpha_q, \alpha_{p^+}), \alpha_{\hat{z}}+\beta_{\hat{z}} = (q,p^+,\alpha_q+\beta_q,\alpha_{p^+}+\beta_{p^+}), -\beta_{\hat{z}} = (q,p^+,-\beta_q, -\beta_{p^+})$.
\end{defn}

\subsubsection{Discrete interconnections via Dirac structures}

With these definitions in place, we now have the tools to state the main result.

\begin{thm}
Given two discrete Dirac structures $D_{\Delta_{Q_1}}^{d+}$ and $D_{\Delta_{Q_2}}^{d+}$ generated from $\Delta_{Q_1}\subset TQ_1$ and $\Delta_{Q_2}\subset TQ_2$ and an interconnection distribution $\Sigma_Q\subset TQ = T(Q_1\times Q_2)$,
\begin{equation}
(D_{\Delta_{Q_1}}^{d+}\oplus D_{\Delta_{Q_2}}^{d+})\boxtimes_d D_{\text{int}}^{d+} = D_{\Delta_Q}^{d+}
\end{equation}
for $\Delta_Q = (\Delta_{Q_1}\oplus \Delta_{Q_2})\cap \Sigma_Q$.
\end{thm}
Thus, the statement $(X_d^k,\mathfrak{D}^+L_d(q_k,q_k^+))\in (D_{\Delta_{Q_1}}^{d+} \oplus D_{\Delta_{Q_2}}^{d+})\boxtimes_d D_{\text{int}}^{d+}$ is equivalent to the statement $(X_d^k,\mathfrak{D}^+L_d(q_k,q_k^+))\in D_{\Delta_Q}^{d+}$ and to the interconnected equations given in \eqref{int_eqns_1}--\eqref{int_eqns_end}.

\begin{proof}
First we recall the definition of a (+) discrete induced Dirac structure,
\begin{equation}
\begin{aligned}
D_{\Delta_Q}^{d+} &= \{((z,z^+),\alpha_{\hat{z}}) \in (T^*Q\times T^*Q)\times T^*(Q\times Q^*) \ | \\
&\qquad (z,z^+) \in \Delta_{T^*Q}^{d+}, \alpha_{\hat{z}} - \Omega_{d+}^\flat(z,z^+)\in \Delta_{Q\times Q^*}^\circ\}.
\end{aligned}
\end{equation}
For $Q = Q_1\times Q_2$ we can write $z = (z_1,z_2)$, $z^+ = (z_1^+, z_2^+)$, $\hat{z} = (\hat{z}_1, \hat{z}_2)$, and $\alpha_{\hat{z}} = (\alpha_{\hat{z}_1}, \alpha_{\hat{z}_2}).$ Then, 
\begin{equation}
\begin{aligned}
D_{\Delta_{Q_1}}^{d+}\oplus D_{\Delta_{Q_2}}^{d+} &= \{((z,z^+),\alpha_{\hat{z}}) \ | \ ((z_1,z_1^+),\alpha_{\hat{z}_1}) \in D_{\Delta_{Q_1}}^{d+} \\
&\qquad \text{and} \ ((z_2,z_2^+),\alpha_{\hat{z}_2})\in D_{\Delta_{Q_2}}^{d+}\}
\end{aligned}
\end{equation}
and
\begin{multline}
(D_{\Delta_{Q_1}}^{d+}\oplus D_{\Delta_{Q_2}}^{d+})\boxtimes_d D_{\text{int}}^{d+} = \{((z,z^+),\alpha_{\hat{z}}) \ | \ \exists \beta_{\hat{z}} \in T_{\hat{z}}^*(Q\times Q^*) \  \text{with}\\  
((z,z^+),\alpha_{\hat{z}}+\beta_{\hat{z}}) \in D_{\Delta_{Q_1}}^{d+}\oplus D_{\Delta_{Q_2}}^{d+}, ((z,z^+),-\beta_{\hat{z}})\in D_{\text{int}}^{d+}\}.
\end{multline}
From the first condition we have $((z_1,z_1^+),\alpha_{\hat{z}_1}+\beta_{\hat{z}_1})\in D_{\Delta_{Q_1}}^{d+}$ and $((z_2,z_2^+),\alpha_{\hat{z}_2}+\beta_{\hat{z}_2})\in D_{\Delta_{Q_2}}^{d+}$.

Consider the distribution conditions first. The distribution condition for $D_{\Delta_Q}^{d+}$ is that $(z,z^+)\in\Delta_{T^*Q}^{d+}$. We can break the distribution condition down as
\begin{equation}
\begin{aligned}
((q,p),(q^+,p^+)) &\in \{((q,p),(q^+,p^+)) \in T^*Q \times T^*Q \ | \ (q,q^+) \in \Delta_Q^{d+}\} \\
&= \{((q,p),(q^+,p^+)) \ | \ (q,q^+) \in (\Delta_{Q_1}^{d+} \oplus \Delta_{Q_2}^{d+})\cap \Sigma_{Q}^{d+}\}\\
&= \{((q,p),(q^+,p^+)) \ | \ (q,q^+)\in \Sigma_{Q}^{d+}, (q_1,q_1^+)\in D_{\Delta_{Q_1}}^{d+}, \\ 
&\qquad\text{and} \ (q_2,q_2^+)\in D_{\Delta_{Q_2}}^{d+}\}.
\end{aligned}
\end{equation}
We now derive the distribution condition from $((z,z^+),\alpha_{\hat{z}}) \in (D_{\Delta_{Q_1}}^{d+}\oplus D_{\Delta_{Q_2}}^{d+})\boxtimes_d D_{\text{int}}^{d+}$. From $((z_1,z_1^+),\alpha_{\hat{z}_1}+\beta_{\hat{z}_1})\in D_{\Delta_{Q_1}}^{d+}$ and $((z_2,z_2^+),\alpha_{\hat{z}_2}+\beta_{\hat{z}_2})\in D_{\Delta_{Q_2}}^{d+}$, we get $(q_1,q_1^+)\in D_{\Delta_{Q_1}}^{d+}$ and $(q_2,q_2^+)\in D_{\Delta_{Q_2}}^{d+}$. From $((z,z^+),-\beta_{\hat{z}})\in D_{\text{int}}^{d+}$ we have $(q,q^+) \in \Sigma_{Q}^{d+}$. Thus, the distribution conditions derived from $((z,z^+),\alpha_{\hat{z}}) \in D_{\Delta_{Q}}^{d+}$ and $((z,z^+),\alpha_{\hat{z}}) \in (D_{\Delta_{Q_1}}^{d+}\oplus D_{\Delta_{Q_2}}^{d+})\boxtimes_d D_{\text{int}}^{d+}$ are equivalent.

Now we consider the second condition, coming from $\alpha_{\hat{z}} - \Omega_{d+}^\flat(z,z^+)\in \Delta_{Q\times Q^*}^\circ$ in the general definition. Recalling the definitions of  $\Delta_{Q\times Q^*}^\circ$, $\hat{z} = (q,p^+)$, and $\Omega_{d+}^\flat(z,z^+) = (q,p^+,p,q^+)$ gives
\begin{equation}
(q,p^+,\alpha_q,\alpha_{p^+}) - (q,p^+,p,q^+) \in \{(q,p,\alpha_q,0)\in T^*(Q\times Q^*) \mid \alpha_qdq\in \Delta_Q^\circ(q)\}.
\end{equation}
In the case of $\Delta_Q = (\Delta_{Q_1}\oplus\Delta_{Q_2})\cap\Sigma_{Q}$ we can rewrite this condition explicitly as
\begin{align}
\alpha_q-p \in [(\Delta_{Q_1}\oplus\Delta_{Q_2})\cap\Sigma_{Q}]^\circ(q_0),\\
\alpha_{p^+} - q^+ = 0.
\end{align}

Now consider the statement $((z,z^+),\alpha_{\hat{z}}) \in (D_{\Delta_{Q_1}}^{d+}\oplus D_{\Delta_{Q_2}}^{d+})\boxtimes_d D_{\text{int}}^{d+}$. First examine $((z,z^+),-\beta_{\hat{z}}) \in D_{\text{int}}^{d+}$. This implies that $-\beta_{\hat{z}} - \Omega_{\text{int}}^{d+}(z,z^+) \in \Sigma_{Q\times Q^*}^\circ$, i.e., $(q,p^+,-\beta_q, -\beta_{p^+}) \in \{ (q,p,\alpha_q, 0) | \alpha_qdq \in \Sigma_{Q}^\circ(q)\}$. Thus, we must have $-\beta_{p^+} = 0$ and $-\beta_q \in \Sigma_{Q}^\circ(q)$. Now we examine $((z,z^+),\alpha_{\hat{z}}+\beta_{\hat{z}}) \in D_{\Delta_{Q_1}}^{d+} \oplus D_{\Delta_{Q_2}}^{d+}$. From the subsection above, we then have that $((z_i,z_i^+),\alpha_{\hat{z}_i}+\beta_{\hat{z}_i})\in D_{\Delta_{Q_i}}^{d+}$ which gives the conditions
\begin{align}
(q_i,p_i^+,\alpha_{q_i}+\beta_{q_i} - p_i, \alpha_{p_i^+} + \beta_{p_i^+} - q_i^+) \in \{(q,p,\alpha,0) | \alpha dq \in \Delta_{Q_i}^\circ\}.
\end{align}
We already know that $\beta_{p^+} = 0$, so these conditions become
\begin{align}
\alpha_{p_i^+} - q_i^+ = 0, \\
\alpha_{q_i} + \beta_{q_i} -p_i \in \Delta_{Q_i}^\circ.
\end{align}
Putting the two indices together gives
\begin{align}
\alpha_{p^+} - q^+ = 0, \\
\alpha_q - p + \beta_q \in \Delta_{Q_1}^\circ\oplus\Delta_{Q_2}^\circ . \label{di_tensordistcond}
\end{align}
We have already established that $\beta_q \in \Sigma_{Q}^\circ(q)$, so \eqref{di_tensordistcond} becomes
\begin{equation}
\alpha_q-p \in (\Delta_{Q_1}^\circ \oplus \Delta_{Q_2}^\circ)(q)\cup \Sigma_{Q}^\circ(q),
\end{equation}
i.e.,
\begin{equation}
\alpha_q - p \in [(\Delta_{Q_1}\oplus\Delta_{Q_2})\cap\Sigma_{Q}]^\circ(q).
\end{equation}
Thus, we derive precisely the same conditions from both $D_{\Delta_Q}^{d+}$ and $(D_{\Delta_{Q_1}}^{d+}\oplus D_{\Delta_{Q_2}}^{d+})\boxtimes_d D_{\text{int}}^{d+}$ and the two structures are equivalent.
\end{proof}

We have now shown that we can interconnect discrete Dirac systems in a way consistent with the variational discretization of the full system and that the Dirac structure preserved by the interconnected discrete system can be viewed as a product of $D_{\Delta_{Q_1}}^{d+}\oplus D_{\Delta_{Q_2}}^{d+}$ with a discrete interaction Dirac structure, analogous to the continuous case. In defining $D_{\text{int}}^{d+}$, we have extended the notion of discrete Dirac structures beyond the induced structures of \cite{LeOh2011}. This extension as well as the definition of $\boxtimes_d$, which is indifferent to whether its operands are induced structures, raises the question of whether we can make a more general definition of discrete Dirac structures for which induced structures are just a special case. We discuss this more in the future work section below.

\section{Numerical Examples}

Continuing the theme of reproducing \cite{JaYo2014} discretely, we now work through the simulation of some of the interconnected examples presented there. We will rehash the setup of each example and then give details of its numerical implementation. In this section we use superscripts to denote coordinates of $q, v$, and $p$ and subscripts to denote numerical time-steps.

\subsection{A chain of spring masses}

The first example is a chain of three spring masses attached to a wall. We consider it to be the interconnection of a chain of two spring masses with the third spring mass pair. Thus, we have two primitive systems with configuration spaces $Q_1= Q_2 = \mathbb{R}^2$. The first system has coordinates $(q^1,q^2)$, the second $(\bar{q}^2, q^3)$. Figures \ref{ChainOfSprings} and \ref{TornSprings} illustrate these two viewpoints. Note that in the torn case we introduce an extra variable, $\bar{q}^2$, to mark the position of the left end of the spring.

\begin{figure}
\begin{center}
\includegraphics[scale=1]{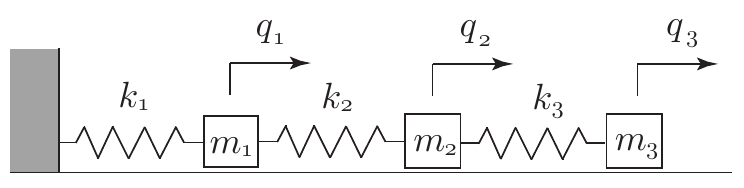}
\end{center}
\caption{A chain of spring masses like that presented in \cite{JaYo2014}.}
\label{ChainOfSprings}
\end{figure}

\begin{figure}
\begin{center}
\includegraphics[scale=1]{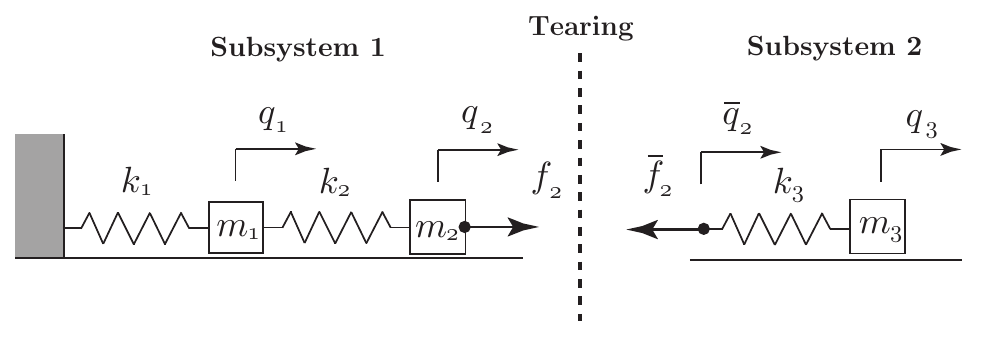}
\end{center}
\caption{The chain of springs as two primitive systems \cite{JaYo2014}.}
\label{TornSprings}
\end{figure}

Viewed separately, the two primitive systems each have the trivial constraint distribution $\Delta_{Q_i} = TQ_i$ and Lagrangians
\begin{equation}
L^1(q^1,q^2,v^1,v^2) = \frac{1}{2}m_1(v^1)^2 + \frac{1}{2}m_2(v^2)^2 - \frac{1}{2}k_1(q^1)^2 - \frac{1}{2}k_2(q^2-q^1)^2
\end{equation}
and
\begin{equation}
L^2(\bar{q}^2,q^3,\bar{v}^2,v^3) = \frac{1}{2}m_3(v^3)^2 - \frac{1}{2}k_3(q^3-\bar{q}^2)^2.
\end{equation}
To interconnect the systems into the chain in Figure \ref{ChainOfSprings}, we need to enforce the constraint $q^2 = \bar{q}^2$. This is a holonomic constraint, but within the framework of Dirac systems we enforce it with a compatible initial condition and a distribution constraint $\Sigma_{\text{int}}(q) = \{v \in T_qQ \ | \ v^2 = \bar{v}^2\}$. Thus, $\Sigma_{\text{int}}^\circ = \Span\{\omega\}$ for $\omega = dq^2 - d\bar{q}^2 \in T_q^*Q$. In coordinates, $\omega = (0,1,-1,0)$.

We discretized both the simple chain of springs in Figure \ref{ChainOfSprings} and the interconnected version described above using the retraction-based methodology laid out in \cite{LeOh2011} and used in the circuit example therein. Namely, we choose the vector space retraction $R_q(v) = q+vh$ for $h$ the timestep, giving $R_{q_k}^{-1}(q_{k+1}) = \frac{1}{h}(q_{k+1}-q_k)$. Then, we set
\begin{align}
L_d^i(q_k^i,q_{k+1}^i) &= hL(q_k^i,R_{i,q_k^i}^{-1}(q_{k+1}^i)), \label{RetractionBasedDiscreteLagrangian}\\
\omega_{d+,i}^a(q_k^i,q_{k+1}^i) &= \langle \omega_i^a(q_k^i), R_{i,q_k^i}^{-1}(q_{k+1}^i) \rangle, \\
\alpha_{d+}^b(q_k, q_{k+1}) &= \langle \alpha^b(q_k), R_{q_k}^{-1}(q_{k+1})\rangle.
\end{align}

For this interconnected system, we then have
\begin{align}
L_d^1(&q_k^1, q_k^2,q_{k+1}^1,q_{k+1}^2) \\
&= h\left[\frac{m_1}{2}\left(\frac{q_{k+1}^1-q_k^1}{h}\right)^2 + \frac{m_2}{2}\left(\frac{q_{k+1}^2-q_k^2}{h}\right)^2 - \frac{k_1}{2}\left(q_k^1\right)^2 - \frac{k_2}{2}\left(q_k^2 - q_k^1\right)^2 \right], \notag \\
L_d^2(&\bar{q}_k^2, q_k^3,\bar{q}_{k+1}^2,q_{k+1}^3) = h\left[\frac{m_3}{2}\left(\frac{q_{k+1}^3 - q_k^3}{h}\right)^2 - \frac{k_3}{2}\left(q_k^3 - \bar{q}_k^2\right)^2\right], \\
\alpha&_{d+}(q_k, q_{k+1}) = \frac{1}{h}\left[\left(q_{k+1}^2-q_k^2\right)-\left(\bar{q}_{k+1}^2 - \bar{q}_k^2\right)\right].
\end{align}
Then the interconnected discrete Dirac equations \eqref{int_eqns_1} through \eqref{int_eqns_end} become
\begin{subequations}
\begin{align}
p_{k+1}^1 &= \frac{m_1}{h}\left(q_{k+1}^1- q_k^1\right), \label{int_springs_1} \\
p_{k+1}^2 &= \frac{m_2}{h}\left(q_{k+1}^2- q_k^2\right), \\
\bar{p}_{k+1}^2 &= 0, \\
p_{k+1}^3 &= \frac{m_3}{h}\left(q_{k+1}^3 - q_k^3\right), \\
p_k^1 &= \frac{m_1}{h}\left(q_{k+1}^1 - q_k^1\right) + hk_1q_k^1 - hk_2\left( q_k^2 - q_k^1\right), \\
p_k^2 &= \frac{m_2}{h}\left(q_{k+1}^2 - q_k^2\right) + hk_2\left(q_k^2 - q_k^1\right) + \lambda, \\
\bar{p}_k^2 &= -hk_3\left(q_k^3 - \bar{q}_k^2\right) - \lambda, \\
p_k^3 &= \frac{m_3}{h}\left(q_{k+1}^3 - q_k^3\right) + hk_3\left(q_k^3 - \bar{q}_k^2\right), \\
0 &= \frac{1}{h}\left[\left(q_{k+1}^2-q_k^2\right)-\left(\bar{q}_{k+1}^2 - \bar{q}_k^2\right)\right]. \label{int_springs_n}
\end{align}
\end{subequations}

For comparison, we apply the same discretization to the monolithic system, obtaining
\begin{align}
L_d(&q_k^1, q_k^2,q_k^3q_{k+1}^1,q_{k+1}^2,q_{k+1}^3) = h\Bigg[\frac{m_1}{2}\left(\frac{q_{k+1}^1-q_k^1}{h}\right)^2 + \frac{m_2}{2}\left(\frac{q_{k+1}^2-q_k^2}{h}\right)^2 \\
&+ \frac{m_3}{2}\left(\frac{q_{k+1}^3 - q_k^3}{h}\right)^2 - \frac{k_1}{2}\left(q_k^1\right)^2 - \frac{k_2}{2}\left(q_k^2 - q_k^1\right)^2 - \frac{k_3}{2}\left(q_k^3 - q_k^2\right)^2\Bigg]. \notag
\end{align}
We use the (+) discrete Dirac equations. Here they simplify to the discrete Euler-Lagrange equations,
\begin{subequations}
\begin{align}
p_{k+1}^1 &= \frac{m_1}{h}\left(q_{k+1}^1- q_k^1\right), \label{mono_springs_1}\\
p_{k+1}^2 &= \frac{m_2}{h}\left(q_{k+1}^2- q_k^2\right), \\
p_{k+1}^3 &= \frac{m_3}{h}\left(q_{k+1}^3 - q_k^3\right), \\
p_k^1 &= \frac{m_1}{h}\left(q_{k+1}^1 - q_k^1\right) + hk_1q_k^1 - hk_2\left( q_k^2 - q_k^1\right), \\
p_k^2 &= \frac{m_2}{h}\left(q_{k+1}^2 - q_k^2\right) + hk_2\left(q_k^2 - q_k^1\right) -hk_3\left(q_k^3 - q_k^2\right), \\
p_k^3 &= \frac{m_3}{h}\left(q_{k+1}^3 - q_k^3\right) + hk_3\left(q_k^3 - q_k^2\right). \label{mono_springs_n}
\end{align}
\end{subequations}

Both systems are fully explicit. We set $m_i = k_i = 1$ and solve the equations using Matlab. The initial conditions are
\begin{align*}
q_0^1 = 0, \\
q_0^2 = \bar{q}_0^2 = 1,\\
q_0^3 = 2,\\
p_0^1 = 0,\\
p_0^2 = \bar{p}_0^2 = 0,\\
p_0^3 = 3.
\end{align*}
We solve the system for 1,000 iterations with a time-step of $h=0.01$. Figures \ref{Fig_SpringsPositions} through \ref{Fig_SpringsConstraintDeviation} show the results of this numerical experiment. Figure \ref{Fig_IntSprings_LongTimeEnergy} uses the same parameters, initial conditions, and time-step but runs for $100,000$ iterations. For an explicit system as simple as this one, the added computational work of solving equations \eqref{int_springs_1} - \eqref{int_springs_n} vs. equations \eqref{mono_springs_1} - \eqref{mono_springs_n} scales linearly with the number of dummy variables and constraints needed to specify the interconnection. This illustrates the point made in the introduction that this technique would be most useful in a situation involving many complex components but relatively simple interactions among components. For this particular example the additional work is imperceptible in practice. Over 1,000 runs the monolithic system has an minimum runtime of 0.0028 seconds compared to 0.0029 seconds for the interconnected system.

\begin{figure}
\includegraphics[scale=0.5]{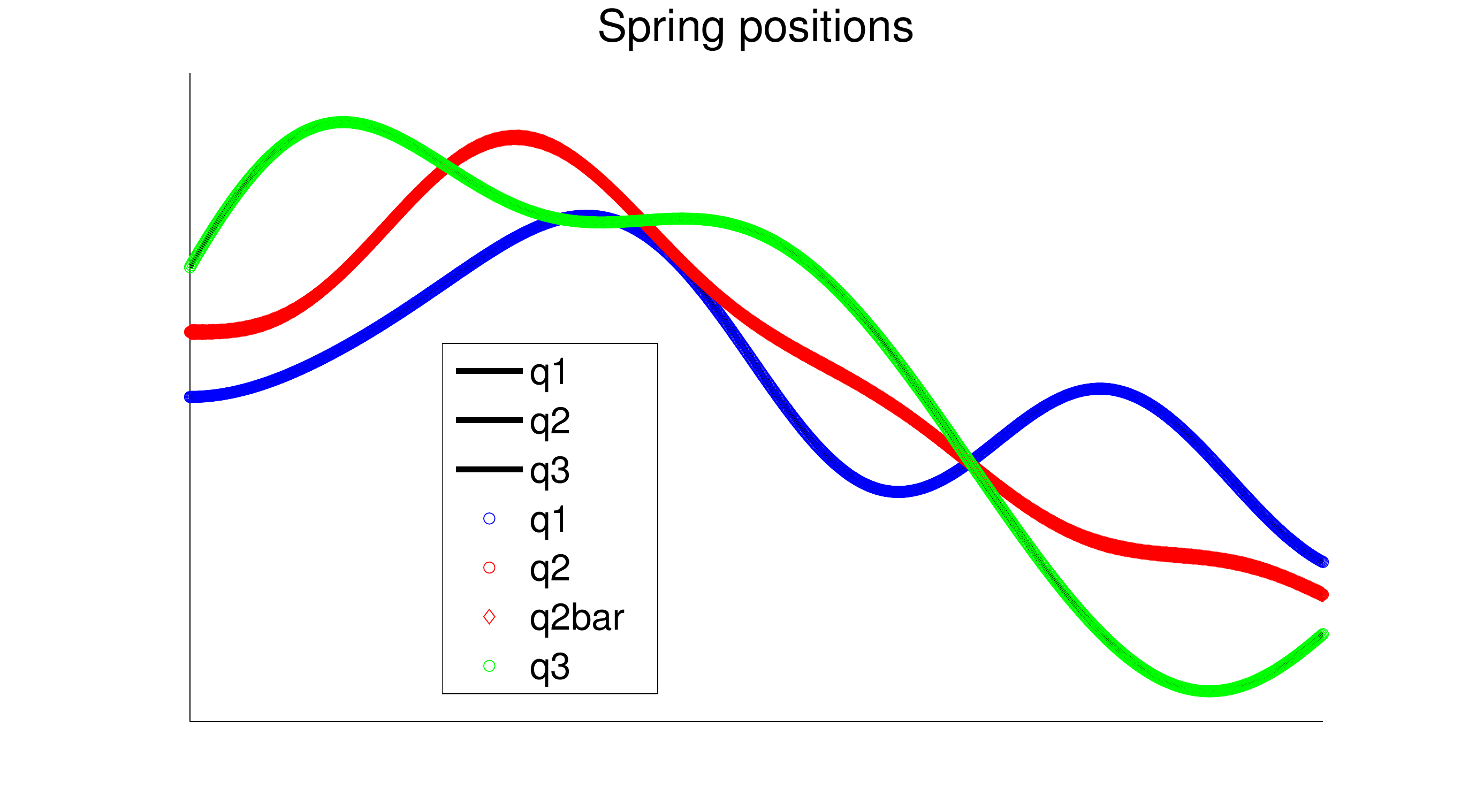}
\caption{A comparison of spring positions over time. Solutions from discretizing the full system are plotted as lines. Solutions from discretizing as two interconnected systems are plotted as hollow shapes. The shapes lie directly over the lines.}
\label{Fig_SpringsPositions}
\end{figure}
                                                                                            
\begin{figure}
\includegraphics[scale=0.5]{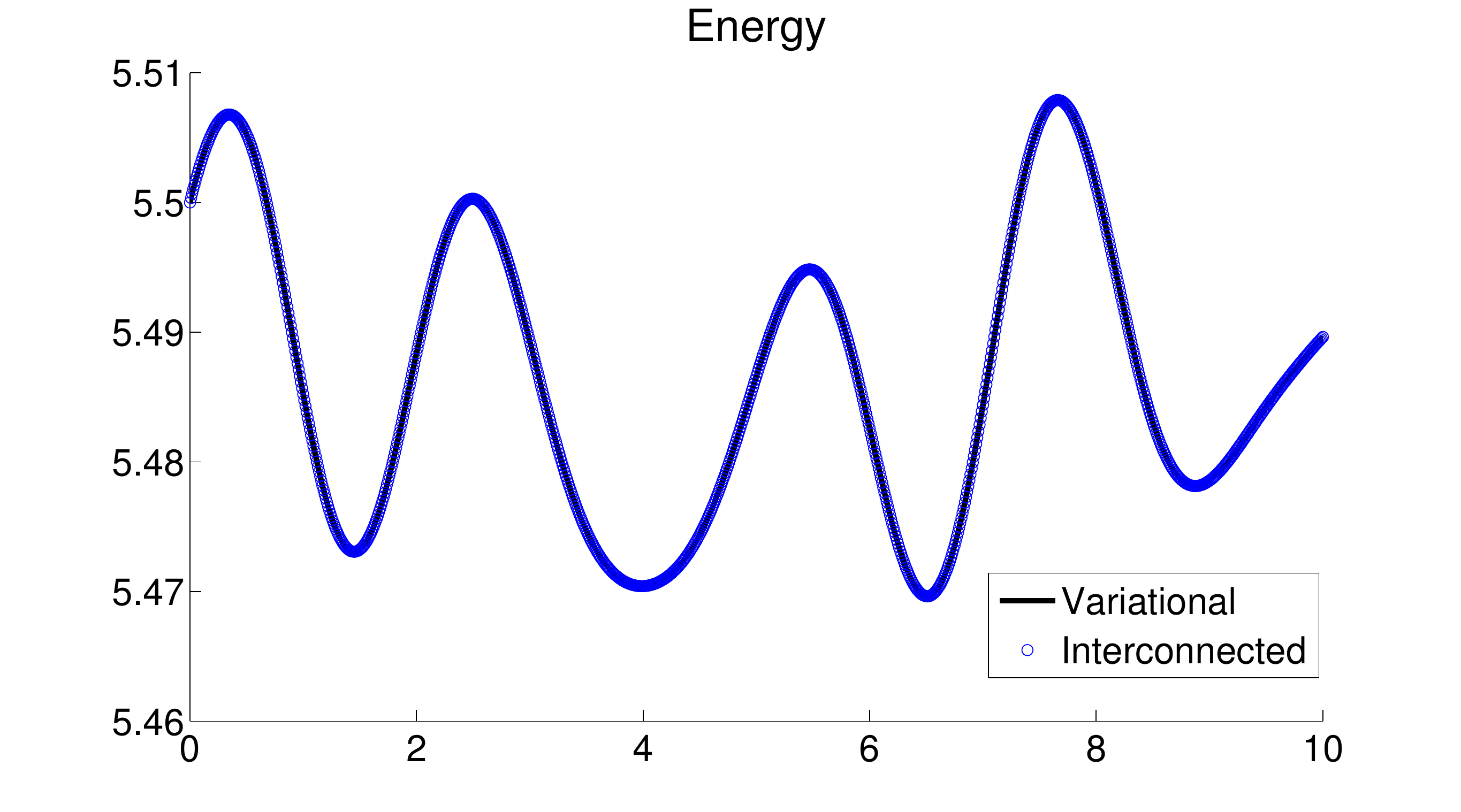}
\caption[A comparison of the spring system energy over time.]{A comparison of the spring system energy over time. The line labeled ``variational" is the energy of the full-system discretization. The hollow circles show the energy for the discretization as two interconnected systems. The hollow circles lie directly on top of the line. Note also the small scale of the vertical axis.}
\label{Fig_SpringsEnergy}
\end{figure}

\begin{figure}
\includegraphics[scale=0.5]{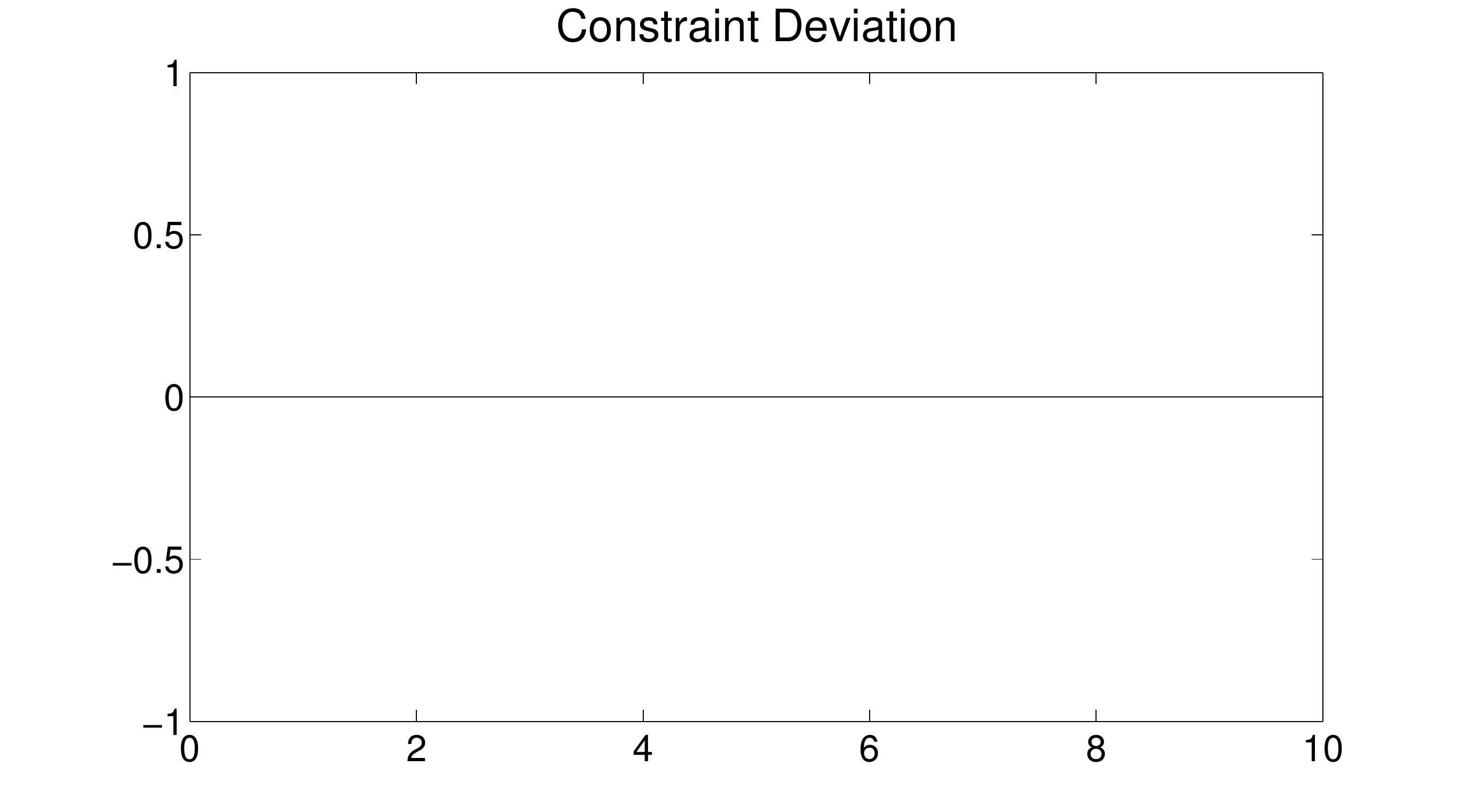}
\caption{Deviation of the interconnected discretization from the constraint $q_2 = \bar{q}_2$ over time. We see that the constraint is preserved to machine precision.}
\label{Fig_SpringsConstraintDeviation}
\end{figure}

\begin{figure}
\includegraphics[scale=0.5]{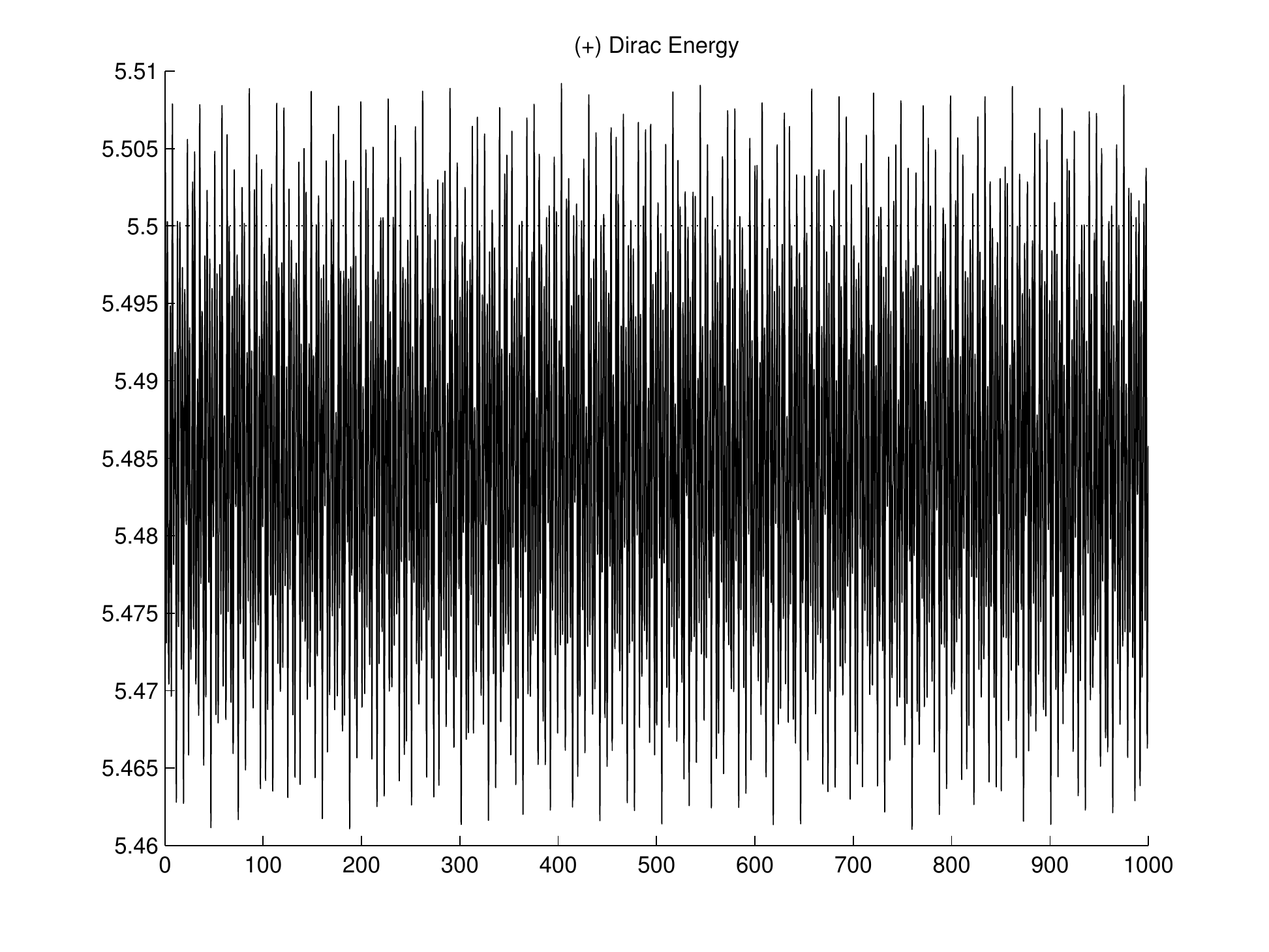}
\caption{The interconnected discretization's energy oscillates over very long times (100,000 iterations at $h=0.01$), much like the energy of classical variational discretization.}
\label{Fig_IntSprings_LongTimeEnergy}
\end{figure}

Figures \ref{Fig_SpringsPositions} and \ref{Fig_SpringsEnergy} compare the interconnected discretization with the discretization of the full system. Note that we are more concerned with reproducing the behavior of the full discretization than with the overall accuracy of the simulation. We have excellent agreement between the two discretizations, with the interconnected results obscuring the full discretization in the figures by lying directly on top. It is also difficult to distinguish in Figure \ref{Fig_SpringsPositions} between the trajectories of $q^2$ and $\bar{q}^2$. This is because, as shown in Figure \ref{Fig_SpringsConstraintDeviation}, the interconnected discretization preserves the $q^2 = \bar{q}^2$ constraint to machine precision. Thus, the trajectories lie atop one another in Figure \ref{Fig_SpringsPositions}. Figure \ref{Fig_SpringsEnergy} shows good agreement between the energy of the full system discretization and that of the interconnected discretization. Lastly, Figure \ref{Fig_IntSprings_LongTimeEnergy} shows that the interconnected discretization exhibits the oscillatory energy behavior characteristic of variational integrators.

\subsection{An LC circuit}

The next example is a very simple parallel RLC circuit which we consider as the joining of a capacitor to the RL loop component. We borrow the illustrations of this idea from \cite{JaYo2014} in Figures \ref{Fig_Circuit} and \ref{Fig_TornCircuit}.

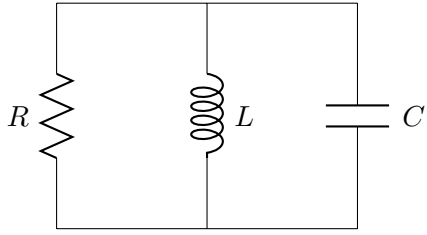
\begin{figure}
\[
\begin{circuitikz}
      \draw (0,0)
      to[R=$R$](0,3)
      to(2,3)
      to[L=$L$] (2,0)
      to(0,0);
      \draw(2,3)
      to(4,3)
      to[C=$C$](4,0)
      to(2,0);
   \end{circuitikz}
\]
\caption{A simple parallel RLC circuit \cite{JaYo2014}.\label{Fig_Circuit}}
\end{figure}

\begin{figure}
\[
\begin{circuitikz}
      \draw (0,0)
      to[R=$R$](0,3)
      to(2,3)
      to[L=$L$] (2,0)
      to(0,0);
      \draw(2,3)
      to(4,3)
      to(4,0)
      to(2,0);
      \draw(4,1.5)
      node[draw, fill=white, fill opacity = 1] at (4,1.5) {$S_1$};
      \draw(2,3.5)
      node[] at (2,3.5) {Primitive circuit 1};
   \end{circuitikz}
   \hspace*{0.25in}
\begin{circuitikz}
      \draw (0,0)
      to(0,3)
      to(2,3)
      to[C=$C$] (2,0)
      to(0,0);
      \draw(0,1.5)
      node[draw, fill=white, fill opacity = 1] at (0,1.5) {$S_2$};
      \draw(2,3.5)
      node[] at (1,3.5) {Primitive circuit 2};
   \end{circuitikz}
\]
\caption{Considering the circuit as two primitive circuits. The $S_i$ boxes show the possible points of connection and represent the influence of any connected circuit components \cite{JaYo2014}.}
\label{Fig_TornCircuit}
\end{figure}
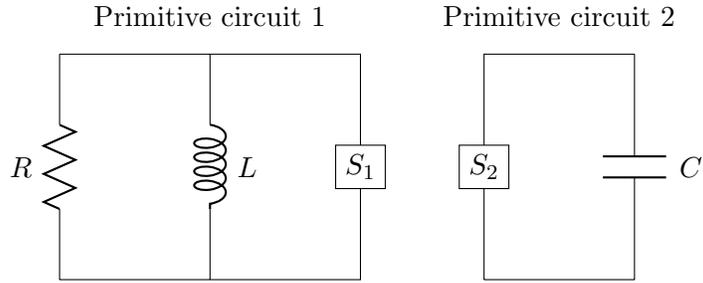

When considering electric circuits as Lagrangian, Hamiltonian or Lagrange--Dirac systems we take the charges as the configuration variables. So the configuration space for the undivided circuit is $\mathbb{R}^3$ with coordinates $(q^R,q^L,q^C)$ representing the charge in the resistor, inductor and capacitor, respectively. Then, $\dot{q}$ represents the currents in each component. The Lagrangian for any circuit is given by the magnetic energy stored in any inductors minus the electric potential energy of any capacitors. For the first primitive circuit we have $Q_1 = \mathbb{R}^3$ with local coordinates $q^1 = (q^R,q^L,q^{S_1})$. The $q^{S_1}$ variable represents the possible point of connection shown in Figure \ref{Fig_TornCircuit} and represents the influence of any connected circuit components. The Lagrangian is just the magnetic energy,
\begin{equation}
L^1(q^1,v_1) = \frac{1}{2}l(v^L)^2,
\end{equation}
where $l$ is the inductance. The circuit has a nontrivial constraint distribution given by Kirchoff's circuit law,
\begin{equation}
\Delta_{Q_1}(q^1) = \{v_1 = (v^R,v^L,v^C) \in T_{q^1}Q_1 \ | \ v^R - v^L-v^{S_1} = 0\}.
\end{equation}
Thus, $\Delta_{Q_1}^\circ = \Span\{\omega_1\}$ for $\omega_1 = dq^R - dq^L + dq^{S_1}$. In coordinates, $\omega_1 = (1,-1,-1)$. This circuit also has an external force due to the resistor, given by $f(q,v) = (q^R,q^L,q^{S_1},-Rv^R,0,0)\in T^*Q_1$.

The second primitive circuit has configuration space $Q_2 = \mathbb{R}^2$ with local coordinates $q_2 = (q^{S_2},q^C)$. Here, the Lagrangian is given by
\begin{equation}
L^2(q_2,v_2) = -\frac{1}{2C}(q^C)^2,
\end{equation}
where $C$ is the capacitance. Again, we have a nontrivial constraint coming from circuit laws,
\begin{equation}
\Delta_{Q_2}(q_2) = \{ v_2 = (v^{S_2},v^C)\in T_{q_2}Q_2 \ | \  v^C - v^{S_2} = 0\}.
\end{equation}
Hence, $\Delta_{Q_2}^\circ = \Span\{\omega_2\}$ for $\omega_2 = -dq^{S_2} + dq^C = (-1,1)$.

To interconnect the two circuits, we set $Q = Q_1\times Q_2, L = L^1 + L^2$ and use
\begin{equation}
\Sigma_{\text{int}} = \{(v^R,v^L,v^{S_1},v^{S_2},v^C) \in TQ \ | \ v^{S_1} = v^{S_2}\}.
\end{equation}

Again, we want to discretize the system according to the retraction-based method of \cite{LeOh2011}. This yields the discrete Lagrangians and constraints for the interconnected system,
\begin{align}
L_d^1(q_k^R,q_k^L,q_k^{S_1},q_{k+1}^R,q_{k+1}^L,q_{k+1}^{S_1}) = \frac{hl}{2}\left(\frac{q_{k+1}^L - q_k^L}{h}\right)^2, \\
\omega_{d+,1}^1(q_k^R,q_k^L,q_{k+1}^R,q_k^{S_1},q_{k+1}^L,q_{k+1}^{S_1}) = \frac{1}{h}\left[ \left(q_{k+1}^R - q_k^R\right) - \left(q_{k+1}^L - q_k^L\right) + \left(q_{k+1}^{S_1} - q_k^{S_1}\right)\right],\\
L_d^2(q_k^{S_2},q_k^C,q_{k+1}^{S_2},q_{k+1}^C) = -\frac{h}{2C}\left(q_k^C\right)^2, \\
\omega_{d+,2}(q_k^{S_2},q_k^C,q_{k+1}^{S_2},q_{k+1}^C) = \frac{1}{h}\left[ \left(q_{k+1}^C - q_k^C\right) - \left(q_{k+1}^{S_2} - q_k^{S_2}\right)\right], \\
\alpha_{d+}(q^k,q^{k+1}) = \frac{1}{h}\left[\left(q_{k+1}^{S_1} - q_k^{S_1}\right) - \left(q_{k+1}^{S_2} - q_k^{S_2}\right)\right].
\end{align}

To address the force in this system coming from the resistor, we must use the forced discrete Dirac equations, \eqref{ForcedDiscreteDiracFirst}-\eqref{ForcedDiscreteDiracLast}. Chapter two of \cite{helen_thesis} develops these equations.
\begin{subequations}
\begin{align}
0&=\omega_{d+}^a(q_k, q_{k+1}), \label{ForcedDiscreteDiracFirst}\\ 
q_{k+1} &= q_k^+,\\
p_{k+1} &= D_2L_d(q_k,q_k^+) + f_d^+(q_k, q_k^+), \\
p_k&= - D_1L_d(q_k,q_k^+) - f_d^-(q_k, q_k^+) + \mu_a\omega^a(q_k).
\label{ForcedDiscreteDiracLast}
\end{align}
\end{subequations}
These equations can be combined with the interconnected Dirac mechanics of \eqref{int_eqns_1}-\eqref{int_eqns_end} to give (for $m_i$ constraints on sub-system $i$ and $l$ interconnection constraints)
\begin{subequations}
\begin{align}
p_{k+1}^i &= D_2L_d^i(q_k^i,q_{k+1}^i) + f_d^{i,+}(q_k^i,q_{k+1}^i), \label{forced_int_eqns_1}\\
p_k^i &= -D_1L_d^i(q_k^i, q_{k+1}^i) - f_d^{i,-}(q_k^i,q_{k+1}^i)+ \mu_a\omega_{i}^a(q_k^i)+\lambda_b\alpha_i^b(q_k),\\
0&=\omega_{d+,i}^a(q_k^i,q_{k+1}^i), & a&=1,\dots,m_i,\\
0&=\alpha_{d+}^b(q_k,q_{k+1}), & b &= 1,\dots,l. \label{forced_int_eqns_end}
\end{align}
\end{subequations}
Equations \eqref{forced_int_eqns_1} - \eqref{forced_int_eqns_end} clearly simplify to the interconnected Dirac equations in the absence of forces and to the forced Dirac equations in the absence of interconnections. They also work well in practice on this numerical example. We know that forces and constraints are equivalent in continuous mechanics, and we leave as future work a rigorous exploration of such an equivalence at the discrete level.

To arrive at an equivalence-preserving discretization, we interpret the retraction based scheme \eqref{RetractionBasedDiscreteLagrangian} as a quadrature rule and define the ($\pm$) discrete forces using the same rule. This gives
\begin{equation}
\begin{aligned}
f_d^+(q_k,q_{k+1}) &=  hf_L\left(q_k, \frac{q_{k+1}-q_k}{h}\right)\left[\frac{\partial q_{k}}{\partial q_{k+1}}\right] \\
&= 0.
\end{aligned}
\end{equation}
and
\begin{equation}
\begin{aligned}
f_d^-(q_k,q_{k+1}) &=  -hf_L\left(q_k, \frac{q_{k+1}-q_k}{h}\right)\left[\frac{\partial q_{k}}{\partial q_{k}}\right] \\
&= -h\left(-R\left[\frac{q_{k+1}^R-q_k^R}{h}\right],0,0\right) \\
&= \left(R\left[q_{k+1}^R-q_k^R\right],0,0\right).
\end{aligned}
\end{equation}

The (+) discrete forced, interconnected Lagrange-Dirac equations are then
\begin{subequations}
\begin{align}
p_{k+1}^R &= 0, \\
p_{k+1}^L &= \frac{l}{h}\left(q_{k+1}^L - q_k^L\right), \\
p_{k+1}^{S_1} &= 0, \\
p_{k+1}^{S_2} &= 0, \\
p_{k+1}^C &= 0, \\
p_k^R &= -R\left(q_{k+1}^R - q_k^R\right) + \mu_1, \\
p_k^L &= \frac{l}{h}\left(q_{k+1}^L - q_k^L\right) - \mu_1, \\
p_k^{S_1} &= -\mu_1 + \lambda, \\
p_k^{S_2} &= \mu_2 - \lambda, \\
p_k^C &= \frac{h}{C}q_k^C - \mu_2, \\
0 &= \left(q_{k+1}^R - q_k^R\right) - \left(q_{k+1}^L - q_k^L\right) - \left(q_{k+1}^{S_1} - q_k^{S_1}\right), \\
0 &= \left(q_{k+1}^C - q_k^C\right) - \left(q_{k+1}^{S_2} - q_k^{S_2}\right) , \\
0 &= \left(q_{k+1}^{S_1} - q_k^{S_1}\right) - \left(q_{k+1}^{S_2} - q_k^{S_2}\right).
\end{align}
\end{subequations}

For the monolithic system, we have the Lagrangian
\begin{equation}
L(q^R,q^L,q^C,v^R,v^L,v^C) = \frac{l}{2}(v^L)^2 - \frac{1}{2C}(q^C)^2
\end{equation}
and the constraint distribution
\begin{equation}
\Delta_Q(q^R,q^L,q^C) = \{v = (v^R,v^L,v^C) \in T_qQ \ | \ v^R - v^L - v^C = 0 \}.
\end{equation}
The (+) discrete forced Lagrange-Dirac equations are then
\begin{subequations}
\begin{align}
p_{k+1}^R &= 0, \\
p_{k+1}^L &= \frac{l}{h}\left(q_{k+1}^L - q_k^L\right), \\
p_{k+1}^C &= 0, \\
p_k^R &= -R\left(q_{k+1}^R - q_k^R\right) + \mu, \\
p_k^L &= \frac{l}{h}\left(q_{k+1}^L - q_k^L\right) - \mu, \\
p_k^C &= \frac{h}{C}q_k^C - \mu, \\
0 &= \left(q_{k+1}^R - q_k^R\right) - \left(q_{k+1}^L - q_k^L\right) - \left(q_{k+1}^C - q_k^C\right).
\end{align}
\end{subequations}

We set the following parameters and initial conditions
\begin{align*}
R &= 1, \\
l &= 0.75, \\
C &= 3, \\
q_0^R &= q_0^L = q_0^{S_1} = q_0^{S_2} = q_0^C = 0, \\
p_0^R &= p_0^{S_1} = p_0^{S_2} = p_0^C = 0, \\
p_0^L &= 10*l.
\end{align*}
We then compared the two discretizations over 400 interations with time-step $h=0.1$. Figures \ref{Fig_CircuitCapacitorComp} through \ref{Fig_CircuitConstraintDeviation} show the results. Again for this simple example the equations are fully explicit with the added computational work linear in the number of dummy variables needed to express the interconnection. Over 1,000 runs the monolithic system has a minimum runtime of $8.16 \times 10^{-4}$ seconds vs. $8.51 \times 10^{-4}$ seconds for the interconnected system.

\begin{figure}
\includegraphics[scale=0.5]{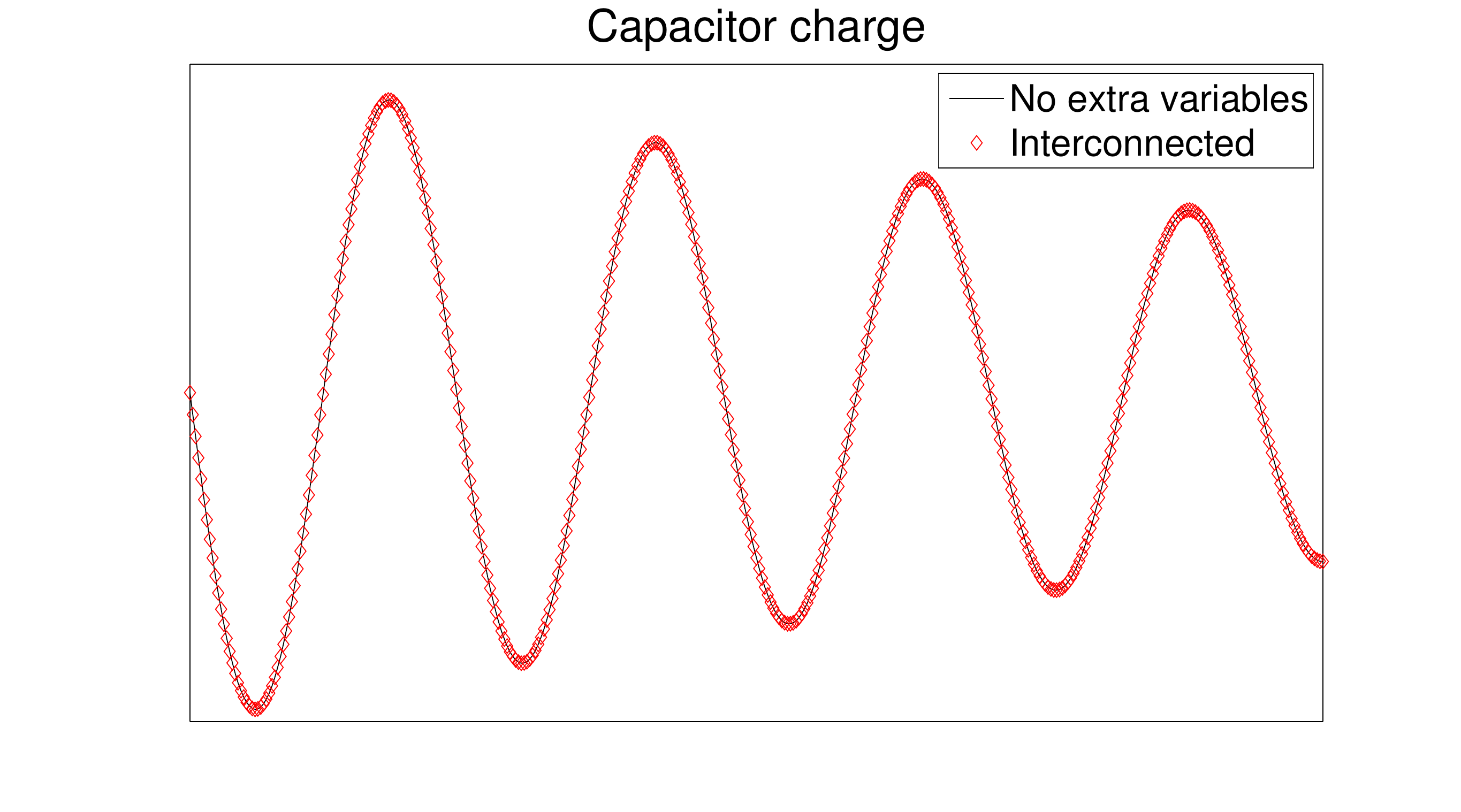}
\caption{A comparison of the capacitor charge in the system generated by the monolithic and interconnected models. The two agree very closely.}
\label{Fig_CircuitCapacitorComp}
\end{figure}

\begin{figure}
\includegraphics[scale=0.5]{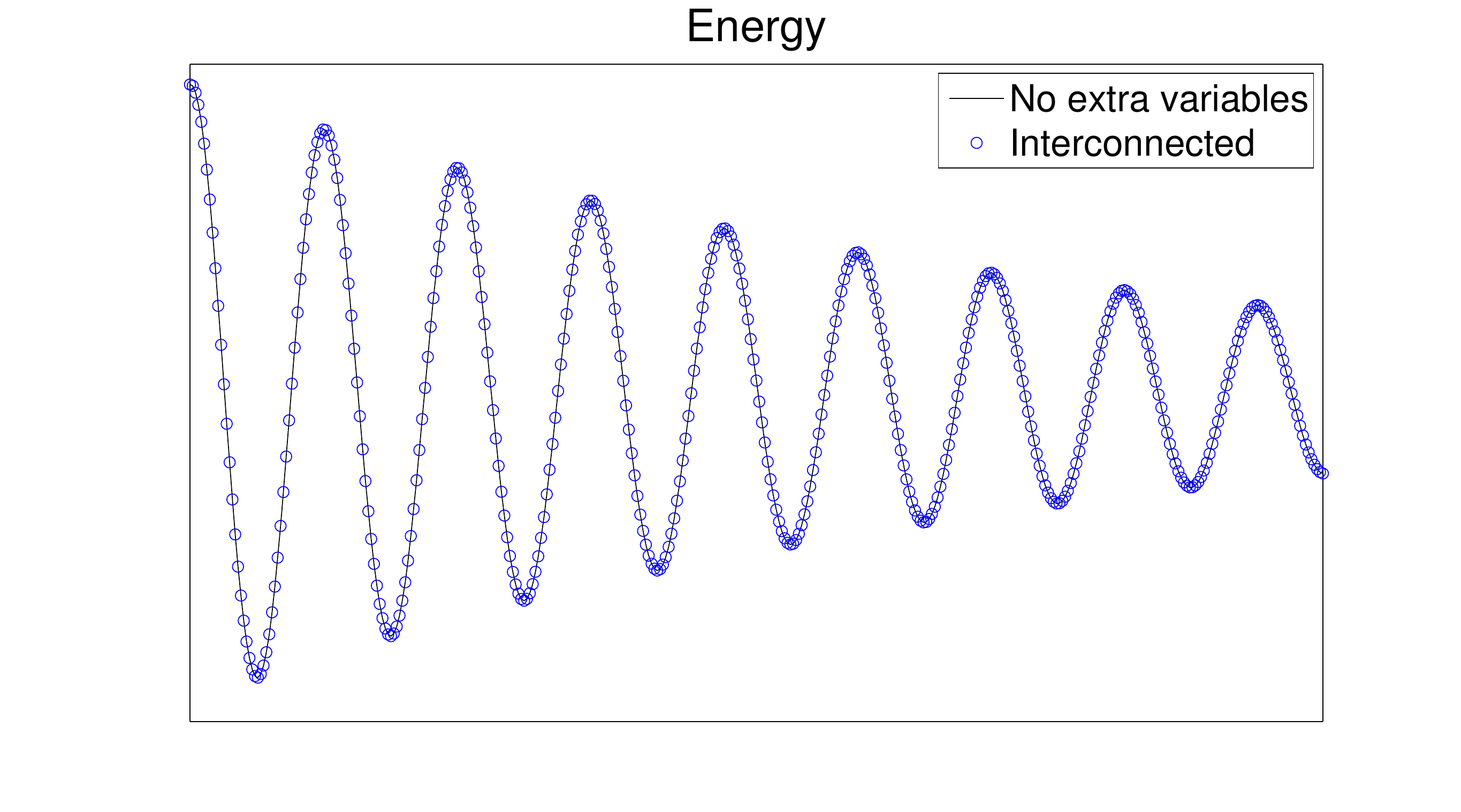}
\caption{The energy in the circuit system generated by the monolithic vs. the interconnected model. The two agree very closely.}
\label{Fig_CircuitEnergyComp}
\end{figure}

\begin{figure}
\includegraphics[scale=0.5]{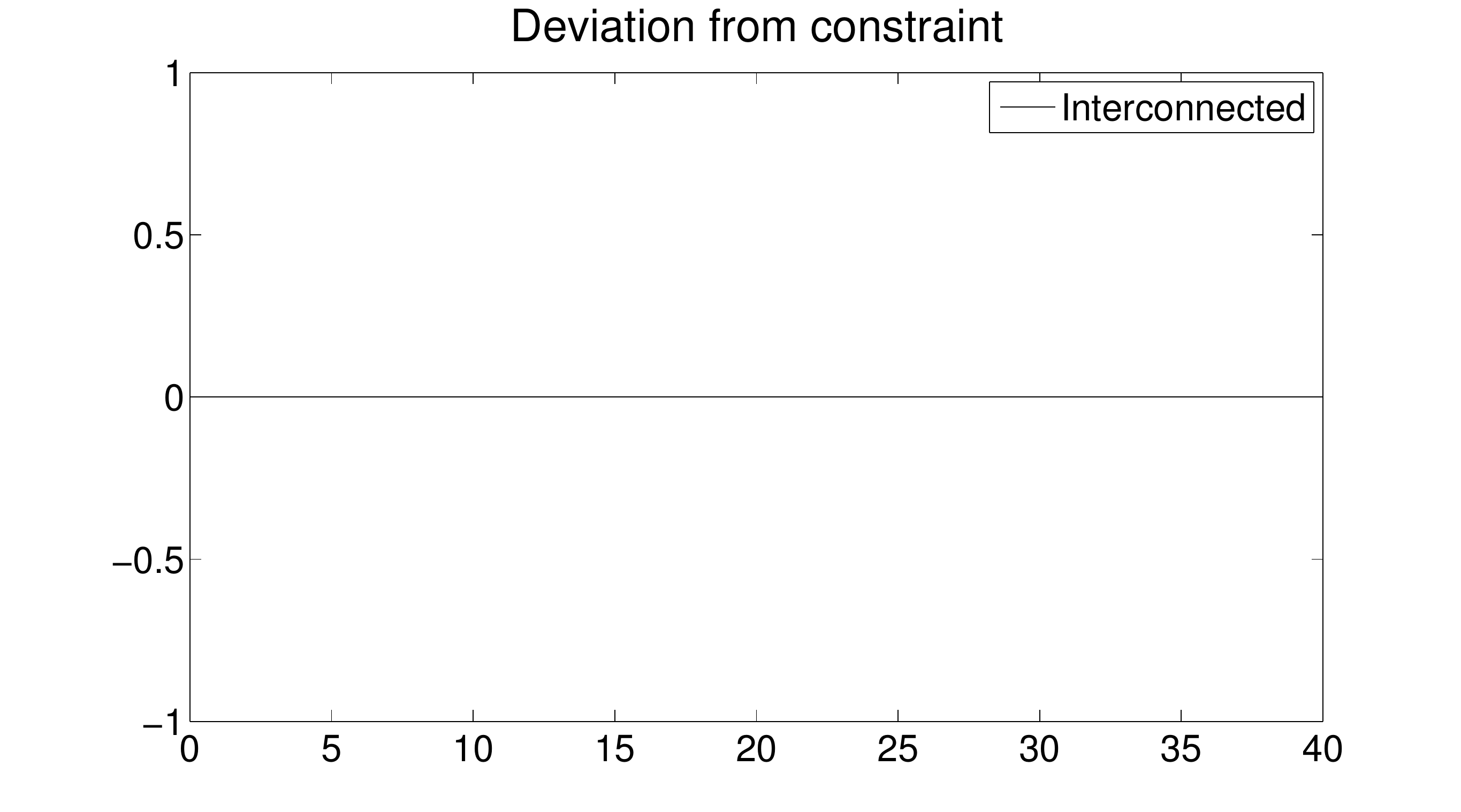}
\caption{The interconnected model preserves the interconnection constraint to machine precision.}
\label{Fig_CircuitConstraintDeviation}
\end{figure}

Figure \ref{Fig_CircuitCapacitorComp} shows that the capacitor charge of the interconnected discretization correctly replicates that of the full system discretization. In Figure \ref{Fig_CircuitEnergyComp}, we see that the same is true for the overall circuit energy. Lastly, Figure \ref{Fig_CircuitConstraintDeviation} shows preservation of the constraint to machine precision in this case as well. Thus, once again, our interconnected discretization behaves equivalently to the full system discretization, as predicted by our theoretical development.

\section{Conclusions and future work}

We have presented a framework for interconnecting discrete Lagrange--Dirac systems, extending the work of \cite{LeOh2011}. Our view of interconnections is based on the perspective presented in \cite{JaYo2014}. In \cite{JaYo2014}, the authors emphasize the equivalence between the constrained view and the interaction force view of interconnections. Our discrete interconnections so far take the constrained point of view. In future work, we would like to see an equivalent interaction-force perspective at the discrete level.

We would also like to further investigate the relationship between discrete Dirac integrators and the vast literature on nonholonomic integrators. With any luck, the two approaches to nonholonomic constraints will mutually shed light on one another.

As a practical consideration, the tearing of systems like those in the examples here can lead to new, redundant variables in the interconnected system. Those extra variables have been dealt with on a case by case basis in this study, and our numerical experiments confirm the the monolithic interconnected system with extra variables produces the same results as the full system without extra variables in these cases. We would of course prefer to have a theoretical justification for introducing and working with extra variables in this way. We leave this as future work.

\section{Acknowledgements}

We gratefully acknowledge helpful comments and suggestions of the referee. HP has been supported by the NSF Graduate Research Fellowship grant number DGE-1144086. ML has been supported in part by NSF under grants DMS-1010687, CMMI-1029445, DMS-1065972, CMMI-1334759, DMS-1411792, DMS-1345013.
\bibliography{ThesisBib}
\bibliographystyle{plainnat}

\end{document}